\documentclass[11pt,reqno]{amsart}

\usepackage{amssymb}
\usepackage{amsthm}
\usepackage{latexsym}
\usepackage{amsfonts}
\usepackage{mathrsfs}
\usepackage{bm}
\usepackage{titlesec}
\usepackage{color}
\usepackage{hyperref}

\newtheorem{thm}{Theorem}[section]
\newtheorem{cor}[thm]{Corollary}
\newtheorem{lem}[thm]{Lemma}
\newtheorem{prop}[thm]{Proposition}

\theoremstyle{definition}
\newtheorem{defn}[thm]{Definition}
\newtheorem{rem}[thm]{Remark}
\newtheorem{ex}[thm]{Example}

\titleformat{\section}{\normalfont\bfseries\centering}{\thesection.}{.25em}{}
\titleformat{\subsection}{\normalfont\bfseries}{\thesubsection.}{.25em}{}
\titlespacing{\section}{0pt}{*3}{*1.5}
\titlespacing{\subsection}{0pt}{*4}{*0.5}
\numberwithin{equation}{section}
\allowdisplaybreaks
\renewcommand{\emptyset}{\varnothing}

\newcommand{\R}{\ensuremath{\mathbb R}}    % Reelle Zahlen
\newcommand{\C}{\ensuremath{\mathbb C}}    % Komplexe Zahlen
\newcommand{\N}{\ensuremath{\mathbb N}}    % Nat"urliche Zahlen
\newcommand{\K}{\ensuremath{\mathbb K}}    % Irgendein Koerper

\newcommand{\calL}{\mathcal L}         
\newcommand{\calM}{\mathcal M}

\newcommand{\calS}{\mathcal S}

\newcommand{\calW}{\mathcal W}         
\newcommand{\calX}{\mathcal X}

\newcommand{\la}{\lambda}
\newcommand{\veps}{\varepsilon}

\newcommand{\linspan}{\operatorname{span}}

\newcommand{\mul}[1]{M(#1)}
\newcommand{\ran}[1]{R(#1)}
\newcommand{\dom}[1]{D(#1)}

\newcommand{\Lra}{\Longrightarrow}
\newcommand{\Sra}{\Rightarrow}

\newcommand{\ol}{\overline}
\newcommand{\ds}{\dotplus}

\newcommand{\rank}{\operatorname{rank}}
\newcommand{\diag}{\operatorname{diag}}

\newcommand{\rmref}[1]{{\rm\ref{#1}}}

\newcommand{\ra}{\rightarrow}

\DeclareMathOperator{\rev}{rev}
\begin{document}
%%%%%%%%%%%%%%%%%%%%%%%%%%%%%%%%%%%%%%%%%%%%%%%%%%%%%%%%%%%%%%%%%%%%%%%%%%%%%
%%%
%%%  HEAD OF PAPER
%%%
\title[Finite Rank Perturbations of Linear Relations and
Matrix Pencils]{Finite Rank Perturbations of Linear Relations and
Matrix Pencils}

\author[L.\ Leben]{Leslie Leben}
\address{{\bf L. Leben:} Institut f\"ur Mathematik, Technische Universit\"at Ilmenau, Postfach 10 05 65, 98684 Ilmenau, Germany}
\email{leslie.leben@gmx.de}

\author[F.\ Mart\'{\i}nez Per\'{\i}a]{Francisco Mart\'{\i}nez Per\'{\i}a}
\address{{\bf F.\ Mart\'{\i}nez Per\'{\i}a:} %Departamento de Matem\'{a}tica
Centro de Matem\'atica de La Plata -- Facultad de Ciencias Exactas, Universidad Nacional de La Plata, C.C.\ 172, (1900) La Plata, Argentina \\ and Instituto Argentino de Matem\'{a}tica "Alberto P. Calder\'{o}n" (CONICET), Saavedra 15 (1083) Buenos Aires, Argentina }
\email{francisco@mate.unlp.edu.ar}

\author[F.\ Philipp]{Friedrich Philipp}
\address{{\bf F. Philipp:} Institut f\"ur Mathematik, Technische Universit\"at Ilmenau, Postfach 10 05 65, 98684 Ilmenau, Germany}
\email{friedrich.philipp@tu-ilmenau.de}
\urladdr{www.tu-ilmenau.de/analysis/team/friedrich-philipp/}

\author[C.\ Trunk]{Carsten Trunk}
\address{{\bf C. Trunk:} Institut f\"ur Mathematik, Technische Universit\"at Ilmenau, Postfach 10 05 65, 98684 Ilmenau, Germany\\ and Instituto Argentino de Matem\'{a}tica "Alberto P. Calder\'{o}n" (CONICET), Saavedra 15 (1083) Buenos Aires, Argentina}
\email{carsten.trunk@tu-ilmenau.de}
\urladdr{www.tu-ilmenau.de/analysis/team/carsten-trunk/}

\author[H.\ Winkler]{Henrik Winkler}
\address{{\bf H. Winkler:} Institut f\"ur Mathematik, Technische Universit\"at Ilmenau, Postfach 10 05 65, 98684 Ilmenau, Germany}
\email{henrik.winkler@tu-ilmenau.de}
\urladdr{www.tu-ilmenau.de/analysis/team/henrik-winkler/}

\thanks{L.\ Leben, F.\ Mart\'{\i}nez Per\'{\i}a, and C.\ Trunk  gratefully acknowledge the support of the DAAD from funds of the German Bundesministerium f\"{u}r Bildung und Forschung (BMBF), Projekt-ID: 57130286. F.\ Mart\'{\i}nez Per\'{\i}a, and C.\ Trunk  gratefully acknowledge the support of the DFG
(Deutsche Forschungsgemeinschaft)  from the project TR 903/21-1.
 In addition, F.\ Mart\'{\i}nez Per\'{\i}a gratefully acknowledges the support from the grant  PIP CONICET 0168. L.\ Leben gratefully acknowledges the support from Carl-Zeiss-Stiftung. F. Philipp gratefully thanks MinCyT Argentina for their support under grant PICT-2014-1480 and the Carl-Zeiss-Stiftung for supporting him within the project {\it DeepTurb - Deep Learning in and of Turbulence}.}

\dedicatory{Dedicated to Henk de Snoo on the occasion of his 75$^{\rm th}$ birthday}

%%%%%%%%%%%%%%%%%%%%%%%%%%%%      ABSTRACT      %%%%%%%%%%%%%%%%%%%%%%%%%%%%%
\begin{abstract}
We elaborate on the deviation of the Jordan structures of two linear relations that are finite-dimensi\-o\-nal perturbations of each other. We compare their number of Jordan chains of length at least $n$. In the operator case, it was recently proved that the difference of these numbers is independent of $n$ and is at most the defect between the operators. One of the main results of this paper shows that in the case of linear relations this number has to be multiplied by $n+1$ and that this bound is sharp. The reason for this behavior is the existence of singular chains.

We apply our results to one-dimensional perturbations of singular and regular matrix pencils. This is done by representing matrix pencils via linear relations. This technique allows for both proving known results for regular pencils as well as new results for singular ones.
\end{abstract}
%%%%%%%%%%%%%%%%%%%%%%%%%%%%%%%%%%%%%%%%%%%%%%%%%%%%%%%%%%%%%%%%%%%%%%%%%%%%%

%\subjclass[2010]{Primary 34B25, 34L15; Secondary 47E05, 47B50}

%\keywords{Sturm-Liouville equation, indefinite weight, non-real eigenvalue}

\maketitle
%\thispagestyle{empty}

%\tableofcontents

\section[\hspace*{1cm}Introduction]{Introduction}

Given a pair of matrices $E, F\in \mathbb{C}^{d\times d}$, the associated matrix pencil is defined by
\begin{equation}\label{pencil177}
P(s):= sE-F.
\end{equation}
The theory of matrix pencils occupies an increasingly important place in linear algebra, due to its numerous applications. For instance, they appear in a natural way in the study of differential-algebraic equations of the form:
\begin{equation}\label{DAE}
E\dot{x}=Fx, \ \ \ x(0)=x_0,
\end{equation}
which are a generalization of the abstract Cauchy problem, see e.g.\ \cite[Chapter 12, \S 7]{G59}.
Substituting $x(t)=x_0e^{st}$ into \eqref{DAE} leads to
\[
(sE - F)x_0 = 0.
\]
Hence, solutions of the above eigenvalue equation for the matrix pencil \eqref{pencil177}
correspond to solutions of the Cauchy problem \eqref{DAE}.
%\eqref{DAE} has a nontrivial solution provided the matrix polynomial $\det(sE - F)$ is not identically zero and $x_0 \neq 0$ satisfies $(sE - F)x_0 = 0$.

The matrix pencil $P$ is called regular if $\det(sE - F)$ is not identically zero, and it is called singular otherwise. Perturbation theory for regular matrix pencils $P(s):= sE-F$ is a well
developed field, we mention here only \cite{DMT08,G,MMRR09,T80} which is a short list of papers devoted to this subject. %We will describe a typical result.
As an example, we describe a well-known result. Recall that for a matrix pencil $P$ as in \eqref{pencil177}, an ordered family of vectors $(x_n, \ldots ,x_{0})$ is a Jordan chain of length $n+1$ at
$\lambda\in\mathbb C$ if $x_{0}\neq 0$ and
\begin{equation*}
 (F -\la E)x_0=0,\quad (F-\lambda E)x_1 = E x_0,\quad \ldots, \quad (F-\lambda E)x_n = E x_{n-1}.
\end{equation*}
Denote by $\mathcal L_{\lambda}^l(P)$ the subspace spanned by the elements of all Jordan chains up to length $l$ at
the eigenvalue $\lambda \in \mathbb C$. If $l=0$ or if $\la$ is not an eigenvalue of $P$ we define $\mathcal L_{\lambda}^l(P)=\{0\}$.
%\[
%\mathcal L_{\lambda}^l(\mathcal A):=\Big\{g_{j}\in\C^d ~|~ 0\leq j\leq l-1,~ \text{$\{g_j,\ldots,g_0\}$ is a Jordan chain at $\lambda$}\Big\}.
%\]
If $P(s)$ is regular and if  $Q(s)$ is a rank-one pencil such that $(P+Q)(s)$ is also regular then for  $n\in\N\cup\{0\}$
the following inequality holds:
	\begin{align*}
	\left|\dim\frac{\mathcal{L}_{\lambda}^{n+1}(P+Q)}{\mathcal{L}_{\lambda}^n(P+Q)}
-\dim\frac{\mathcal{L}_{\lambda}^{n+1}(P)}{\mathcal{L}_{\lambda}^{n}(P)}\right|\leq 1.
	\end{align*}
In this form it can be found in \cite{G}, but it is mainly due to \cite{DMT08} and \cite{T80}.
%is shown in \cite{G} (which is a consequence of \cite{DMT08}and \cite{T80}):
The proof of this inequality, as many other results concerning perturbation theory for regular matrix pencils, is based on a detailed analysis of the determinant.

Perturbation theory for singular matrix pencils is studied only in a few papers so far. Roughly speaking, it started with
the investigation of the Kronecker canonical form of a fixed singular matrix pencil $P$ under low rank perturbations in \cite{DD07}. There, the generic change in the Kronecker canonical form of a singular pencil under
low-rank perturbations resulting again in a singular pencil
is considered. In this case the term generic
refers to the fact that the perturbations are from an open dense subset of the set of  pencils with fixed sizes and rank, cf.\ \cite[Theorem 3.1]{DD07}.
In \cite{HMP,MMW17} the effect of generic regularizing perturbations was considered, i.e.\
perturbations whose rank is exactly the difference of full rank and the rank of a singular pencil.
While the focus in \cite{MMW17} is on symmetric rank-one
perturbations, \cite{HMP} contains the general
low-rank case. In \cite{MMW15} the rank-one distance to singularity  as the smallest norm of a rank-one perturbation that makes a given pencil singular is expressed as a quadratic constrained optimization problem.

Finally, we would like to mention that in a recent manuscript \cite{BR20} the authors characterize the Kronecker structure of a matrix pencil obtained by a rank-one perturbation of another matrix pencil in terms of the homogenous invariant
factors and the row and column minimal indices of the original and the perturbed pencil via transforming
it in a matrix pencil completion problem.

Here we develop a different approach to treat finite rank perturbations of singular
matrix pencils. This is done by representing matrix pencils
via %subspaces or, what is the same,
linear relations, see also \cite{BennByer01,BB06,BTW}.
The classical philosophy to treat linear multi-valued mappings or relations was just to concentrate on the operator part and getting rid of the multi-valued part by projection.
At this place one has to mention the particular contributions of Henk de Snoo to linear relations, who started, together with many coauthors, a seminal work on this subject.
%At the beginning were the publications \cite{DS74,DS87,DS87-2}, were the authors treated linear relations as subspaces in product spaces.
The publications \cite{DS74,DS87,DS87-2} are among the first where the authors treated linear relations as subspaces in product spaces.  Later on, Henk de Snoo was involved in investigations where linear relations arise in a natural way in extension and perturbation theory \cite{DHMS06,HSSW19, HSSW07, HSSW17} for many kinds of linear operators or relations, see also \cite{LSSW10,LSSW12}. Concerning his contributions to the structure of linear relations, see \cite{ss09, ssw05,ssw}. Of course, this is a non-exhaustive list of Henk de Snoo's publications about this topic.
%Of course, there are much more publications about this topic with him.

Each matrix $E\in\mathbb C^{d\times d}$ is considered as a linear relation
via its graph, i.e.\ the subspace of $\mathbb C^d\times \mathbb C^d$ consisting of pairs of the form $\{x, Ex\}$, $x\in \mathbb C^d$. Also, the inverse $E^{-1}$ (in the sense of linear relations) of a non-necessarily invertible matrix $E$ is the subspace of $\mathbb C^d\times \mathbb C^d$ consisting of pairs of the form $\{Ex, x\}$, $x\in \mathbb C^d$.
%of all tuples  $\left(\begin{smallmatrix} x\\ Ex\end{smallmatrix}\right)$
 %in $\mathbb C^d\times \mathbb C^d$.
Multiplication of linear relations is defined in analogy to
multiplication of matrices, see Section \ref{Preliminaries} for the details. Then, to a matrix pencil
$P(s)=sE-F$ we associate the linear relation $E^{-1}F$.
%In the sense of linear relations,
%the inverse  $E^{-1}$ of a non-invertible matrix $E$ is given as the linear relation of all tuples  $\left(\begin{smallmatrix} Ex\\ x\end{smallmatrix}\right)$ in $\mathbb C^d\times \mathbb C^d$
%and we obtain
%$$
%E^{-1}F = N [F; \ -E].
%$$

There exists a well developed spectral theory for linear relations, see e.g.\ \cite{Arens,c,ssw05}.
An eigenvector at $\lambda \in \C$ of $E^{-1}F$ is a tuple of the form $\{x,\lambda x\} \in E^{-1}F$, $x\ne 0$. Jordan chains are defined in a similar way, see Section \ref{Linear Relations} below.

In Section \ref{Nicki} we show that (point) spectrum and Jordan chains of
$E^{-1}F$ coincide with (point) spectrum and Jordan chains of the matrix pencil $P$ in \eqref{pencil177}, respectively.
This is the key to translate spectral properties of a matrix pencil to its associated linear relation
and vice versa. The advantage of this approach is that it is applicable not only to regular matrix pencils, but  also to singular matrix pencils.
%Notice that in this case all complex numbers are eigenvalues (as the determinant is identically equal to zero).

Given a matrix pencil $P$ as in \eqref{pencil177}, we consider
one-dimensional perturbations of the form
\begin{equation*}%\label{stoerung}
 Q(s)=w(su^*-v^*), %\quad w\neq 0,
\end{equation*}
where $u,v,w\in \mathbb C^d$, $(u,v) \ne (0,0)$ and $w\neq 0$.
Then $P$ and $P+Q$ are rank-one perturbations of each other, which means that they differ by a rank-one matrix polynomial.
Recall that the rank of a matrix pencil $P$ is the largest
$r\in\mathbb{N}$ such that $P$, viewed as a matrix with polynomial entries, has minors of size $r$  that are not identically zero   \cite{DMT08,G59}.
As described above, to the matrix pencils  $P$ and $P+Q$
there correspond the linear relations $E^{-1}F$ and $\left(E+wu^*\right)^{-1}(F+wv^*)$, respectively,
which turn out to be one-dimensional perturbations of each other, see Section \ref{Osterei}.
Then, the main result of this paper is Theorem \ref{grosserWurf} below.
It consists of the following perturbation estimates for
singular (and regular) matrix pencils:
\begin{itemize}
\item[\rm (i)]
 If $P$ is regular but $P+Q$ is singular, then %we have for $k\in\N\setminus\{0\}$
	\begin{align*}
-1-n \leq  \dim\frac{\mathcal{L}_{\lambda}^{n+1}(P+Q)}{\mathcal{L}_{\lambda}^n(P+Q)}-
\dim\frac{\mathcal{L}_{\lambda}^{n+1}(P)}{\mathcal{L}_{\lambda}^{n}(P)}
\leq 1.
	\end{align*}
\item[\rm (ii)]
 If $P$ is singular and $P+Q$ is regular, then %we have for $k\in\N\setminus\{0\}$
	\begin{align*}
-1 \leq  \dim\frac{\mathcal{L}_{\lambda}^{n+1}(P+Q)}{\mathcal{L}_{\lambda}^n(P+Q)}-
\dim\frac{\mathcal{L}_{\lambda}^{n+1}(P)}{\mathcal{L}_{\lambda}^{n}(P)}
\leq n+1.
	\end{align*}
\item[\rm (iii)]
 If both $P$ and $P+Q$ are singular, then %we have for $k\in\N\setminus\{0\}$
	\begin{align*}
\left|\dim\frac{\mathcal{L}_{\lambda}^{n+1}(P+Q)}{\mathcal{L}_{\lambda}^n(P+Q)}-
\dim\frac{\mathcal{L}_{\lambda}^{n+1}(P)}{\mathcal{L}_{\lambda}^{n}(P)}\right|
\leq n+1.
	\end{align*}
\end{itemize}
Later, in Section \ref{KCF}, we explain how to interpret this result in terms of the Kronecker invariants associated to the Kronecker canonical forms of the matrix pencils $P$ and $P+Q$.

Theorem \ref{grosserWurf} follows from the corresponding result for
one-dimensional perturbations of linear relations, which is
the second main result of this paper. It is the content of
Sections \ref{Linear Relations} and \ref{Osterei}, which is of independent interest.
 More precisely,
given linear relations $A$ and $B$ in a linear space $X$ which are one-dimensional perturbations of each other, we show  that $N(A^{n+1})/N(A^n)$ is finite-dimensional if and only if $N(B^{n+1})/N(B^n)$ is finite-dimensional and, in this case,
\begin{equation}\label{abschaetzung2222}
\left|\dim\frac{N(B^{n+1})}{N(B^n)} - \dim\frac{N(A^{n+1})}{N(A^n)}\right|\,\le\,n+1.
\end{equation}
Here $N(A)$ denotes the kernel of the linear relation $A$, that is, the set of all $x\in X$
such that $\{x,0\}\in A$. If, in addition, $A\subset B$ or $B\subset A$, we show that the left-hand side in
\eqref{abschaetzung2222} is bounded by $n$.
However, in Section \ref{s:sharpness} we show that the bound in \eqref{abschaetzung2222}
is sharp. It is worth mentioning that if $A$ and $B$ are linear operators in $X$ the left-hand side in \eqref{abschaetzung2222} is bounded by $1$, see \cite{blmt}.

In Section \ref{Osterhase} we extend the above result to $p$-dimensional perturbations.
In this case, we show that the left-hand side in
\eqref{abschaetzung2222} is bounded by $(n+1)p$.
Again, this estimate improves to $np$ in case that $A\subset B$ or $B\subset A$, and to $p$ if
$A$ and $B$ are operators, cf.\ \cite{blmt}.

\section[\hspace*{1cm}Preliminaries]{Preliminaries}
\label{Preliminaries}

Throughout this paper $X$ denotes a vector space over $\K$, where $\K$ stands for the real field $\R$ or the complex field $\C$. Each subspace $W$ of $X$ determines an equivalence relation in $X$, we say that $x\in X$ is congruent to $y\in X$ if $x-y\in W$. Then, we denote by $X/W$ or $\tfrac{X}{W}$ the set of all equivalence classes of $X$ with respect to this equivalence relation. $X/W$ is also a vector space over $\K$, which is called the {\em quotient space of $X$ over $W$}, see e.g.\ \cite{N66}.

Elements (pairs) from $X\times X$ will be denoted by $\{x,y\}$, where $x,y\in X$. A {\em linear relation} in $X$ is a linear subspace of $X\times X$. Linear operators can be treated as linear relations via their graphs: each linear operator $T:\dom T\ra X$ in $X$,
where $\dom T$ stands for the domain of $T$, is identified with its graph
\[
\Gamma(T):=\left\{ \{x,Tx\}:\ x\in \dom T\right\}.
\]
For the basic notions and properties of linear relations we refer to
\cite{Arens,c,H06}. However,
we follow here  the above mentioned approach proposed in \cite{DS74,DS87,DS87-2}.
%In some sense, A.\ Dijksma and H.\ de Snoo
%shifted the way to treat linear relations away from multi-valued mapping to subspaces
%in product spaces, here we refer to their papers \cite{DS74,DS87,DS87-2} and,
%in addition, see \cite{ssw05,ssw, ss09}.

We denote the domain and the range of a linear relation $A$ in $X$ by $\dom A$ and $\ran A$, respectively,
\[
 \dom A = \left\{ x\in X \;:\; \exists\, y:\ \{x,y\} \in A \right\} \quad \mbox{and} \quad
 R(A)=\left\{ y\in X \;:\; \exists\, x:\ \{x,y\} \in A \right\}.
\]
Furthermore, $N(A)$ and $\mul A$ denote the \textit{kernel} and the
\textit{multivalued} part of $A$,
\[
 N(A)=\left\{ x\in X \;:\; \{x,0\} \in A \right\} \quad \mbox{and} \quad
\mul A=\left\{y\in X \;:\; \{0,y\}\in A \right\}.
\]
Obviously, a linear relation $A$ is the graph of an operator if and only if $\mul
A=\{0\}$. The inverse $A^{-1}$ of a linear relation $A$ always exists and is given by
\begin{equation}\label{Dammtor}
A^{-1} =\left\{ \{y,x\}\in X \times X  \;:\;  \{x,y\}
\in A \right\}.
\end{equation}

We recall that the product of two linear relations $A$ and $B$ in $X$ is defined as
\begin{equation*}
AB = \left\{ \{x, z\}\;:\;  \{ y, z\} \in A \mbox{ and } \{ x,y\} \in B \;\mbox{ for some } y \in X\right\}.
\end{equation*}
As for operators the product of linear relations is an associative operation. We denote
$A^0 := I$, where $I$ denotes the identity operator in $X$, and for $n=1,2,\ldots$ the $n$-th power of $A$ is defined recursively by
\begin{equation*}
A^n := AA^{n-1}.
\end{equation*}
Thus, we have $\{x_n, x_0\}\in A^n$ if and only if there exist $x_1,\ldots,x_{n-1}\in X$ such that
\begin{equation}\label{Minaj}
\{x_n, x_{n-1}\}, \{x_{n-1}, x_{n-2}\}, \ldots, \{x_1, x_0\}\in A.
\end{equation}
In this case, \eqref{Minaj} is called a {\it chain of} $A$. We also use the
shorter notation $(x_n,\ldots,x_0)$.

For a linear relation $T$ in $X$ and $m\in\N$, consider the vector space of $m$-tuples of elements in $T$:
$$
T^{(m)} := \underbrace{T\times T\times\dots\times T}_{\text{$m$ times}}\ ,
$$
and also the space of $m$-tuples of elements in $T$ which are chains of $T$:
\begin{align}\label{chains}
\calS_m^T := \big\{\left(\{x_{m},x_{m-1}\},\ldots,\{x_1,x_0\}\right) : (x_{m},x_{m-1},\ldots,x_0)\text{ is a chain of $T$}\big\}.
\end{align}
Clearly, $\calS_m^T$ is a subspace of $T^{(m)}$.

\begin{lem}\label{l:villa_crespo}
Let $A$ and $C$ be linear relations in $X$ such that  $C\subset A$ and $\dim(A/C) = 1$. Then for each $m\in\N$ the following inequality holds:
\begin{equation}\label{e:chain_diff}
\dim(\calS_m^A/\calS_m^C)\,\le\,m.
\end{equation}
\end{lem}
\begin{proof}
We make use of Lemma 2.2 in \cite{abpt} which states that whenever $M_0,N_0,M_1,N_1$ are subspaces of a linear space $\calX$ such that $M_0\subset M_1$ and $N_0\subset N_1$, then
$$
\dim\frac{M_1\cap N_1}{M_0\cap N_0}\,\le\,\dim\frac{M_1}{M_0} + \dim\frac{N_1}{N_0}.
$$
With this lemma the proof of \eqref{e:chain_diff} is straightforward. Indeed, since $\calS_m^C = \calS_m^A\cap C^{(m)}$, we obtain from the lemma and from $\dim(A/C) = 1$ that
$$
\dim(\calS_m^A/\calS_m^C) = \dim\,\frac{\calS_m^A\cap A^{(m)}}{\calS_m^A\cap C^{(m)}}\,\le\,\dim(A^{(m)}/C^{(m)}) = m,
$$
which is \eqref{e:chain_diff}.
\end{proof}

\medskip
For relations $A$ and $B$ in $X$ the {\em operator-like sum} $A+B$ is the relation defined by
\[
A+B = \left\{ \{ x,y+z\} \;:\;  \{x,y\} \in A,  \{x,z\} \in B \right\}.
\]
%and for $\lambda\in\C$ one defines $\lambda A =\left\{ \{x,\lambda y\} \;:\; \{x,y\}\in A \right\}$.
%Hence, we have
%\[
%A-\lambda = \left\{ \{x,y-\lambda x\} \;:\; \{x,y\}\in A \right\}.
%\]
The notions of eigenvalue, root manifolds and point spectrum also apply to linear relations. Given $\lambda\in\C$, $A-\lambda$ stands for the linear relation $A-\lambda I$:
\begin{align*}
A-\lambda = \left\{ \{x,y-\lambda x\} \ :\  \{x,y\}\in A \right\}.
\end{align*}
Then, $\lambda\in\C$ is an \emph{eigenvalue of} $A$ if $N(A-\lambda)\neq\{0\}$. On the other hand, we say that $A$ has an \emph{eigenvalue at} $\infty$ if $\mul{A}\neq\{0\}$. The point spectrum of $A$ is the set $\sigma_p(A)$ consisting of the eigenvalues $\lambda\in\C\cup\{\infty\}$ of $A$.

A chain $(x_n,\ldots,x_0)$ of $A$ is called a {\it quasi-Jordan chain of} $A$ {\it
at zero} (or simply a {\it quasi-Jordan chain of $A$} if $x_0\in N(A)$. If  $(x_n,\ldots,x_0)$ is a
quasi-Jordan chain of $A$, then $x_j \in N(A^{j+1})$ for $j=0, \ldots, n$.
If, in addition, $x_n\in M(A)$ and $(x_n,\ldots,x_0)\neq (0,\ldots,0)$, then the chain is called a {\em singular chain} of $A$.
The tuple $(x_n,\ldots,x_0)$ is called a {\it quasi-Jordan chain of $A$ at $\lambda\in\C$}, if $(x_n,\ldots,x_0)$ is a quasi-Jordan chain of the linear relation $A-\lambda$. The tuple $(x_n,\ldots,x_0)$ is called a quasi-Jordan chain of $A$ at $\infty$, if $(x_n,\ldots,x_0)$ is a quasi-Jordan chain at zero of $A^{-1}$.
Note that we admit linear dependence (and even zeros) within the elements of a quasi-Jordan chain.

\medskip
We reserve the notion of a Jordan chain of a linear relation for a particular situation which is discussed in the next section.

\section[\hspace*{1cm}
Linear independence of Jordan chains]{Linear independence of quasi-Jordan chains}\label{Linear Relations}

In what follows only quasi-Jordan chains at zero are considered, so we call them simply quasi-Jordan chains.
Assume that $T$ is a linear operator in $X$ and consider $x_0,\ldots,x_n\in \dom T$ such that
$$
Tx_0 =0 \quad \mbox{and} \quad Tx_j=x_{j-1}, \mbox{ for all } 1 \leq j\leq n.
$$
Then $\{x_n, x_{n-1}\}, \{x_{n-1}, x_{n-2}\}, \ldots, \{x_0, 0\}\in \Gamma(T)$. So, if we consider $T$
also as a linear relation via its graph, $(x_n,\ldots,x_0)$ is a quasi-Jordan chain of $T$.

As $T$ is a linear operator, it is well-known that the following facts are equivalent:
\begin{itemize}
\item[\rm (i)] $x_0 \ne 0$.
\item[\rm (ii)] The set of vectors $\{x_n,\ldots,x_0\}$ is linearly independent in $X$.
\item[\rm (iii)] $[x_n]\ne 0$, where $[x_n]$ is the equivalence class in
$N(T^{n+1})/N(T^n)$.
\item[\rm (iv)] $[x_j]\ne 0$ for all $1\leq j \leq n$, where $[x_j]$ is the equivalence class in
$N(T^{j+1})/N(T^j)$.
\end{itemize}
Therefore, if $T$ is a linear operator and $x_0\neq 0$, $(x_n,\ldots,x_0)$ is a quasi-Jordan chain of the linear relation $\Gamma(T)$ if and only if it is a Jordan chain at zero of the linear operator $T$ in the usual sense.

%\end{rem}
\medskip

However, the four statements above are no longer equivalent for
linear relations which contain singular chains, see the following example.

\begin{ex}
Let $x_0$ and $x_1$ be two linearly independent elements of $X$ and let
$$
A:= \mbox{span}\,\left\{ \{0,x_0\}, \{x_0,0\}, \{x_1,x_0\}\right\}.
$$
Then $x_0\neq 0$ but $(0,x_0)$ is a quasi-Jordan chain with linear dependent entries, hence
the equivalence of (i) and (ii) from above does not hold.

Moreover, $(x_1,x_0)$ is a quasi-Jordan chain with linearly independent entries.
But, as  $\{x_1,x_0\}$ and $\{0,x_0\}$ are both elements of $A$, due to linearity, also
$\{x_1,0\}$ is an element of $A$ and, hence, $[x_1]=0$ in $N(A^2)/N(A)$, i.e.\ (iii) is not satisfied. Therefore, conditions (ii) and (iii) are neither equivalent for linear relations.
\end{ex}

As it was mentioned before, the situation shown in the example is a consequence of the existence of singular chains in the relation $A$, or equivalently, the presence of vectors in the intersection of the kernel of $A$ and the multivalued part of $A^n$ for some $n\in\N$.
For arbitrary linear relations we have the following equivalence.

\begin{prop}\label{Francisquito}
Let $A$ be a linear relation in $X$ and $(x_n,\ldots,x_0)$ be a quasi-Jordan chain of $A$. Then the following statements are equivalent:
\begin{itemize}
\item[\rm (i)] $x_0 \notin M(A^n)$.
\item[\rm (ii)] $[x_n]\ne 0$, where $[x_n]$ is the equivalence class in
$N(A^{n+1})/N(A^n)$.
\item[\rm (iii)] $[x_j]\ne 0$ for all $1\leq j \leq n$, where $[x_j]$ is the equivalence class in
$N(A^{j+1})/N(A^j)$.
\end{itemize}
In particular, if any of the three equivalent statements holds, then the vectors
$x_0, \ldots, x_n$ are linearly independent in $X$.
\end{prop}
\begin{proof}
Since $(x_n,\ldots,x_0)$ is a quasi-Jordan chain of $A$, we have that
\begin{equation}\label{chain}
\{x_n, x_{n-1}\},\ldots,\{x_1,x_0\},\{x_0,0\}\in A.
\end{equation}
We show that (i) and (ii) are equivalent. If  $x_0\in M(A^n)$, then there exist $y_1,\ldots, y_{n-1}\in X$ such that
\[
\{0, y_{n-1}\}, \ldots, \{y_2,y_1\}, \{y_1, x_0\}\in A.
\]
Subtracting this chain from the one in \eqref{chain} we end with
\[
\{x_n, x_{n-1}-y_{n-1}\}, \ldots, \{x_2-y_2, x_1-y_1\}, \{x_1-y_1, 0\} \in A.
\]
Thus, $x_n\in N(A^n)$, or equivalently, $[x_n]=0$.
Conversely,
if $[x_n]=0$ then $x_n\in N(A^n)$. Hence, there exist $u_1,\ldots, u_{n-1}\in X$ such that
\[
\{x_n, u_{n-1}\}, \ldots, \{u_2,u_1\}, \{u_1,0\}\in A.
\]
Taking the difference of   \eqref{chain} and the chain above  we obtain
\[
\{0, x_{n-1}-u_{n-1}\}, \ldots, \{x_2-u_2,x_1-u_1\}, \{x_1-u_1, x_0\}\in A,
\]
i.e.\ $x_0\in M(A^n)$.

Now we show that (ii) and (iii) are equivalent. Obviously (iii) implies (ii).
Hence, assume $[x_n]\ne 0$. Then, by (i), $x_0 \notin M(A^n)$. But as
$M(A^j) \subset M(A^n)$ for all $1\leq j \leq n$, we have
$x_0\notin M(A^j)$ for all $1\leq j \leq n$. Applying (ii) to every $[x_j]$ we obtain (iii).

It remains to show the additional statement concerning the linear independence of the
vectors $x_0, \ldots, x_n$. This is the case if the equation
$\sum_{j=0}^n\alpha_jx_j=0$ implies that all $\alpha_j$, $j=0, \ldots, n$,
are equal to $0$. By (iii) we see that all $x_j$ are non-zero. If not all $\alpha_j$
are equal to $0$, let $n_0$ be the largest index $j$ with $\alpha_{j} \ne 0$. It follows that
$$
x_{n_0} = - \alpha_{n_0}^{-1} \sum_{j=0}^{n_0 -1}\alpha_jx_j \in N(A^{n_0}),
$$
hence $[x_{n_0}]=0$, in contradiction to (iii).
\end{proof}

The above considerations lead to the following definition of a Jordan chain for a linear relation.
\begin{defn}\label{Tunis}
Let $(x_{n},\ldots,x_{0})$ be a quasi-Jordan chain of a linear relation $A$ in $X$.
We call it a \emph{Jordan chain at zero of length} $n+1$ \emph{in} $A$ if
$$
[x_n] \ne 0 \mbox{ in } N(A^{n+1})/N(A^n).
$$
Moreover, $(x_{n},\ldots,x_{0})$ is called a \emph{Jordan chain at} $\la\in\C$ {of length} $n+1$ \emph{in} $A$ if
it is a Jordan chain at zero of $A-\la$ and a \emph{Jordan chain at} $\infty$ {of length} $n+1$ \emph{in} $A$ if
it is a Jordan chain at zero of $A^{-1}$.

\end{defn}

We remark that our Definition \ref{Tunis} is equivalent to the definition formulated in \cite{ssw05}
but different from the one used in \cite{BTW},
 where the term Jordan chain was used for an object which is here called
 quasi-Jordan chain together with the assumption
 that all elements of the quasi-Jordan chain are linearly independent.

In the sequel we will make use of the following lemma.
\begin{lem}\label{l:Encuentro}
Let $A$ be a linear relation in $X$ and let $(x_{k,n},\ldots,x_{k,0})$, $k=1,\ldots,m$, be $m$
quasi-Jordan chains of $A$. Then
\begin{equation*}%\label{Tolosa}
\dim\linspan\{[x_{1,n}],\ldots,[x_{m,n}]\} = \dim \frac{\calL}{\calL\cap M(A^n)},
\end{equation*}
where $\calL := \linspan\{x_{1,0},\ldots,x_{m,0}\}$.
\end{lem}
\begin{proof}
Given $m$ quasi-Jordan chains of $A$ as in the statement, consider the following linear transformations
\begin{align*}
T : \K^m\to\frac{N(A^{n+1})}{N(A^n)},\qquad & Tu := \sum_{k=1}^mu_k[x_{k,n}],\quad  u=(u_1,\ldots,u_m)\in\K^m, \ \ \text{and
}\\
S : \K^m\to N(A),\qquad  & Su := \sum_{k=1}^mu_kx_{k,0},\quad  u=(u_1,\ldots,u_m)\in\K^m.
\end{align*}
On one hand, observe that $R(T)=\linspan\{[x_{1,n}],\ldots,[x_{m,n}]\}$ and $R(S)=\calL$.
On the other hand, we have that
\begin{equation*}
N(T) =  \left\{ u\in\K^m \;:\; Su \in M(A^n) \right\}
\end{equation*}
Indeed, $u=(u_1,\ldots,u_m)\in N(T)$ if and only if $\left[\sum_{k=1}^m u_k x_{k,n}\right] = 0$ which, by Proposition \ref{Francisquito}, is equivalent to $Su=\sum_{k=1}^mu_kx_{k,0}\in M(A^n)$.

In particular,
\begin{align*}
\dim N(T) &= \dim  \left\{ u\in\K^m \;:\; Su \in M(A^n) \right\}  = \dim N(S) + \dim \calL\cap M(A^n),
\end{align*}
and the rank-nullity theorem yields
\begin{align*}
\dim\linspan\{[x_{1,n}],\ldots,[x_{m,n}]\} &= \dim R(T)= m- \dim N(T) \\ &= m-(\dim N(S) + \dim \calL\cap M(A^n)) \\
&= \dim \calL- \dim \calL\cap M(A^n)= \dim \frac{\calL}{\calL\cap M(A^n)},
%(m-\dim N(S)) - \dim \calL\cap M(A^n)=
\end{align*}
where we have used that $R(S)=\calL$.
\end{proof}

In the following we will study linear independence of quasi-Jordan chains.

\begin{lem}\label{l:JC_li}
Let $(x_{k,n},\ldots,x_{k,0})$, $k=1,\ldots,m$, be $m$ quasi-Jordan chains of a linear relation $A$ in $X$. Consider the following statements:
\begin{enumerate}
\item[{\rm (i)}]   The set $\{[x_{1,n}],\ldots,[x_{m,n}]\}$ is linearly independent in $N(A^{n+1})/N(A^n)$.
\item[{\rm (ii)}]  The set $\{x_{k,j} : k=1,\ldots,m, j=0,\ldots,n\}$ is linearly independent in $X$.
\item[{\rm (iii)}] The set  of pairs
$$
\{\{x_{k,j},x_{k,j-1}\} : k=1,\ldots,m, j=1,\ldots,n\}\cup\{\{x_{k,0},0\} : k=1,\ldots,m\}
$$
is linearly independent in $A$.
\end{enumerate}
Then the following implications hold: $\text{\rm (i)}\;\Lra\;\text{\rm (ii)}\;\Lra\;\text{\rm (iii)}$.
If, in addition,
\begin{align*}
\linspan\{x_{1,0},\ldots, x_{m,0}\}\cap \mul{A^n}=\{0\},
\end{align*}
holds, then the three conditions $\text{\rm (i)}$, $\text{\rm (ii)}$, and $\text{\rm (iii)}$ are equivalent.
\end{lem}
\begin{proof}
The implication (ii)$\Sra$(iii) is straightforward by use of the linear independence of the first components of the pairs in (iii).  Let us prove the implication (i)$\Sra$(ii). Assume that $\{[x_{1,n}],\ldots,[x_{m,n}]\}$ is linearly independent. Let $\alpha_{k,j}\in\K$, $j=0,\ldots,n, k=1,\ldots,m$, such that
\begin{equation}\label{e:linind}
\sum\limits_{j=0}^n \sum\limits_{k=1}^m \alpha_{k,j} x_{k,j} = 0.
\end{equation}
It is easily seen that the following tuple is a quasi-Jordan chain of $A$:
\begin{equation*}
  \left( \sum\limits_{j=0}^n\sum\limits_{k=1}^m \alpha_{k,j} x_{k,j}, \sum\limits_{j=1}^n\sum\limits_{k=1}^m \alpha_{k,j} x_{k,j-1},\ldots, \sum\limits_{j=n-1}^n\sum\limits_{k=1}^m \alpha_{k,j} x_{k,j-n+1}, \sum\limits_{k=1}^m \alpha_{k,n} x_{k,0}  \right).
\end{equation*}
From this and \eqref{e:linind} it follows that $\sum_{k=1}^m \alpha_{k,n} x_{k,0}\in\mul{A^n}$, which, by Proposition \ref{Francisquito}, implies
for equivalence classes in $N(A^{n+1})/N(A^n)$
$$
 \left[\sum\limits_{j=0}^n\sum\limits_{k=1}^m \alpha_{k,j} x_{k,j}\right] =\sum\limits_{k=1}^m \alpha_{k,n} [x_{k,n}]=0.
$$
Hence,
$\alpha_{k,n} = 0$ for $k=1,\ldots,m$ and \eqref{e:linind} reads as
\begin{equation}\label{Lindwurm}
\sum\limits_{j=0}^{n-1} \sum\limits_{k=1}^m \alpha_{k,j} x_{k,j} = 0.
\end{equation}
Now one can construct a quasi-Jordan chain as above starting with
the sum in \eqref{Lindwurm}. Repeating the above argument shows $\alpha_{k,n-1}=0$ for
$k=1,\ldots,m$. Proceeding further in this manner yields (ii), since all $\alpha_{k,j}$ in \eqref{e:linind} are equal to zero.

Now assume that $\linspan\{x_{1,0},\ldots, x_{m,0}\}\cap \mul{A^n}=\{0\}$.
By Lemma \ref{l:Encuentro},
\[
\dim \linspan \{[x_{1,n}],\ldots,[x_{m,n}]\}=\dim \linspan\{x_{1,0},\ldots,x_{m,0}\}.
\]
We have to show that in this case (iii) implies (i). But if we assume (iii), in particular we have that $\{x_{1,0},\ldots,x_{m,0}\}$ is linearly independent. Therefore, $\{[x_{1,n}],\ldots,[x_{m,n}]\}$ is also linearly independent, completing the proof.
%Let $\sum_{k=1}^m \alpha_{k} [x_{k,n}]=0.$ Then it follows by Proposition \ref{Francisquito} that $\sum_{k=1}^m \alpha_{k} x_{k,0}\in\mul{A^n}$,
%hence $\sum_{k=1}^m \alpha_{k} x_{k,0}=0$. Therefore
%$$
%\sum_{k=1}^m \alpha_{k} \{x_{k,0},0\} = \{0,0\}
%$$
%and (iii) implies that $\alpha_{k}=0$ for $k=1, \ldots, m$, which shows (i).
\end{proof}

\section[\hspace*{1cm}One-dimensional perturbations]{One-dimensional perturbations}
\label{Osterei}
The following definition, taken from \cite{ABJT}, specifies the idea of a one-dimensional perturbation for linear relations.

\begin{defn}\label{laplata16}
Let $A$ and $B$ be linear relations in $X$. Then $B$ is called an {\em one-dimensional perturbation} of $A$ (and vice versa) if
\begin{equation*}
\max\left\{\dim\frac{A}{A\cap B},\,\dim\frac B{A\cap B}\right\} = 1.
\end{equation*}
In particular, $A$ is called a {\em one-dimensional extension} of $B$ if $B\subset A$ and $\dim(A/B) = 1$.
\end{defn}

The next lemma describes in which way (quasi-)Jordan chains of a one-dimensional extension $A$ of a linear relation $C$ can be linearly combined to become (quasi-)Jordan chains of $C$.
The proof is based on the following simple principle: If $M$ is a subspace of $N$ and $\dim(N/M) = 1$, then whenever $x,y\in N$, $y\notin M$, there exists some $\la\in\K$ such that $x - \la y\in M$.

\begin{lem}\label{l:shifting}
Let $A$ and $C$ be linear relations in $X$ such that  $C\subset A$ and $\dim(A/C) = 1$. If $(x_{k,n},\ldots,x_{k,0})$,
$k=1,\ldots,m$, are $m$ quasi-Jordan chains of $A$, then after a possible reordering, there exist $m-1$ quasi-Jordan chains $(y_{k,n},\ldots,y_{k,0})$, $k=1,\ldots,m-1$, of $C$  such that
\begin{equation*}
y_{k,j}\in x_{k,j} + \linspan\{x_{m,\ell} : \ell=0,\ldots,j\},\qquad k=1,\ldots,m-1,\,j=0,\ldots,n.
\end{equation*}
Moreover, if $\{[x_{1,n}],\ldots, [x_{m,n}]\}$ is linearly independent in $N(A^{n+1})/N(A^n)$ then the set $\{[y_{1,n}],\ldots, [y_{m-1,n}]\}$ is linearly independent in $N(C^{n+1})/N(C^n)$.

On the other hand, if  the set $\{x_{k,j} : k=1,\ldots,m, j=0,\ldots,n\}$ is linearly independent in $X$
then the set $\{y_{k,j} : k=1,\ldots,m-1, j=0,\ldots,n\}$ is linearly independent in $X$.
\end{lem}

\begin{proof}
For any quasi-Jordan chain $(z_n,z_{n-1},\ldots,z_0)$ of $A$ we agree to write $\hat z_j = \{z_j,z_{j-1}\}$ for $j=1,\ldots,n$ and $\hat z_0 = \{z_0,0\}$. Consider the set
\begin{align*}
J :=  \{(k,j)\in \{1,\ldots,m\}\times\{0,\ldots,n\} : \hat x_{k,j}\notin C\}.
\end{align*}
If $J=\emptyset$ then all $m$ quasi-Jordan chains are in
 $C$ and the proof is completed. Therefore, assume $J\neq\emptyset$. Set
$$
h:= \min \bigl\{j\in\{0,\ldots,n\} : (k,j)\in J\text{ for some }k\in\{1,\ldots,m\}\}.
$$
Choose some $\kappa\in\{1,\ldots,m\}$ such that $(\kappa,h)\in J$. After a reordering of the indices we can assume that $\kappa = m$.

Since $\hat x_{m,h}\notin C$, there exist $\alpha_{k,h}\in\K$, $k=1,\ldots m-1$, such that
$$
\hat x_{k,h} - \alpha_{k,h}\hat x_{m,h}\in C
$$
for $k=1,\ldots m-1$. If $h=n$, we stop here. Otherwise, there exist $\alpha_{k,h+1}\in\K$, $k=1,\ldots m-1$, such that
$$
\hat x_{k,h+1} - \alpha_{k,h}\hat x_{m,h+1} - \alpha_{k,h+1}\hat x_{m,h}\in C
$$
for $k=1,\ldots m-1$. If $h=n-1$, the process terminates. Otherwise, there  exist $\alpha_{k,h+2}\in\K$ such that
$$
\hat x_{k,h+2} - \alpha_{k,h}\hat x_{m,h+2} - \alpha_{k,h+1}\hat x_{m,h+1} - \alpha_{k,h+2}\hat x_{m,h}\in C
$$
for $k=1,\ldots m-1$. We continue with this procedure up to $n$, where in the last step we find $\alpha_{k,n}\in\K$ such that
$$
\hat x_{k,n} - \alpha_{k,h}\hat x_{m,n} - \alpha_{k,h+1}\hat x_{m,n-1} - \ldots - \alpha_{k,n-1}\hat x_{m,h+1} - \alpha_{k,n}\hat x_{m,h}\in C
$$
for $k=1,\ldots m-1$. Summarizing, we obtain numbers $\alpha_{k,j}\in\K$, $k=1,\ldots,m-1$, $j=h,\ldots,n$, such that
$$
\hat u_{k,j} := \hat x_{k,j} - \sum_{i=h}^{j}\alpha_{k,i}\,\hat x_{m,j+h-i}\;\in\;C
$$
for all $k=1,\ldots,m-1$, $j=h,\ldots,n$. We now define
\[
y_{k,j} := x_{k,j} - \sum_{i=h}^{\min\{j+h,n\}}\alpha_{k,i}\,x_{m,j+h-i},
\]
for $k=1,\ldots m-1$ and $j=0,\ldots,n$. For $0\le j < h$ (if possible, i.e., $h>0$),
$$
\hat y_{k,j} = \hat x_{k,j} - \sum_{i=h}^{\min\{j+h,n\}}\alpha_{k,i}\,\hat x_{m,j+h-i}\,\in\,C
$$
is a consequence of the definition of $h$, whereas for $j\ge h$ we also have
$$
\hat y_{k,j} = \hat u_{k,j} - \sum_{i=j+1}^{\min\{j+h,n\}}\alpha_{k,i}\,\hat x_{m,j+h-i}\,\in\,C.
$$
This shows that $(y_{k,n},\ldots,y_{k,0})$ is a quasi-Jordan chain of $C$ for each $k=1,\ldots,m-1$.
From the definition of $y_{k,j}$ we also see that $y_{k,j}\in x_{k,j} +
 \linspan\{x_{m,j},\ldots,x_{m,0}\}$ for all $j=0,\ldots, n$ and $k=1,\ldots,m-1$.

Now, assuming the linear independence of $\{[x_{1,n}],\ldots, [x_{m,n}]\}$ in $N(A^{n+1})/N(A^n)$, we prove the linear independence of $\{[y_{1,n}],\ldots, [y_{m-1,n}]\}$ in $N(C^{n+1})/N(C^n)$.
Since $y_{k,0} = x_{k,0} - \alpha_{k,h}x_{m,0}$ for $k=1,\ldots,m-1$, the linear
independence of $\{y_{1,0},\ldots,y_{m-1,0}\}$ in $X$ easily follows from that of $\{x_{1,0},\ldots,x_{m,0}\}$. Furthermore,
$$
\linspan\{y_{1,0},\ldots,y_{m-1,0}\}\cap M(C^n)\,\subset\,\linspan\{x_{1,0},\ldots,x_{m,0}\}\cap M(A^n),
$$
and the claim follows from Lemma \ref{l:Encuentro}.

Finally, assume that the set $\{x_{k,j} : k=1,\ldots,m, j=0,\ldots,n\}$ is linearly independent.
%We show the last statement of Lemma \ref{l:shifting}.
 Also, let $\beta_{k,j}\in\K$, $k=1,\ldots,m-1$, $j=0,\ldots,n$, such that $\sum_{k=1}^{m-1}\sum_{j=0}^n\beta_{k,j}y_{k,j} = 0$. Then
\begin{align*}
0
&= \sum_{k=1}^{m-1}\sum_{j=0}^n\beta_{k,j}\left(x_{k,j} - \sum_{i=h}^{\min\{j+h,n\}}\alpha_{k,i}\,x_{m,j+h-i}\right)\\
&= \sum_{k=1}^{m-1}\sum_{j=0}^n\beta_{k,j}x_{k,j} - \sum_{j=0}^n\sum_{i=h}^{\min\{j+h,n\}}\left(\sum_{k=1}^{m-1}\beta_{k,j}\alpha_{k,i}\right)x_{m,j+h-i}
\end{align*}
From this, we see that $\beta_{k,j} = 0$ for $k=1,\ldots,m-1$ and $j=0,\ldots,n$. Therefore, the set $\{y_{k,j} : k=1,\ldots,m-1, j=0,\ldots,n\}$ is linearly independent in $X$.
\end{proof}

In the main result of this section, Theorem \ref{t:LaLucila} below, we will compare the dimensions of $N(A^{n+1})/N(A^n)$ and $N(B^{n+1})/N(B^n)$ for two linear relations $A$ and $B$ that are one-dimensional perturbations of each other. To formulate it, we define the following value for two linear relations $A$ and $B$ in $X$ and $n\in\N\cup\{0\}$:
\begin{align}\nonumber
s_n(A,B) := \max\big\{\dim(\calL\cap M(A^n)) : \;&\calL \text{ is a subspace of } N(A\cap B)\cap R((A\cap B)^n),\\ \label{Def_s_n}
&\calL\cap M((A\cap B)^n) = \{0\}\big\}.
\end{align}
%\begin{align}\nonumber
%s_n(A,B) := \max\big\{\dim(\calL\cap M(A^n)) : \;&\calL \text{ is a subspace of } N(C)\cap R(C^n) \text{ and } \calL\cap M(C^n) = \{0\}\big\}.
%\end{align}
The quantity $s_n(A,B)$ can be interpreted as the number of (linearly independent) singular chains of $A$ of length $n$ which are not singular chains of $A\cap B$.
To justify this statement, assume that $s_n(A,B)=r$. Then, denoting $C=A\cap B$, there exists a subspace $\calL$ of $N(C)\cap R(C^n)$ such that $\dim (\calL\cap M(A^n))=r$ and $\calL\cap M(C^n) = \{0\}$. On one hand, if $\{x_{1,0},\ldots, x_{r,0}\}$ is a basis of $\calL\cap M(A^n)$, then each $x_{k,0}$, $k=1,\ldots, r$, determines a quasi-Jordan chain $(x_{k,n},\ldots, x_{k,1},x_{k,0})$ of $C$, because $\calL\subseteq N(C)\cap R(C^n)$. Also, since $\calL\cap M(C^n) = \{0\}$, Lemma \ref{l:JC_li} implies that $\{[x_{1,n}],\ldots,[x_{r,n}]\}$ is linearly independent in $N(C^{n+1})/N(C^n)$. In particular, the quasi-Jordan chains $(x_{k,n},\ldots, x_{k,1},x_{k,0})$ are not singular chains of $C$.
On the other hand, each $x_{k,0}$, $k=1,\ldots, r$, determines a singular chain of $A$ of length $n$ because $x_{k,0}\in M(A^n)\cap N(A)$.
\medskip

Note that we always have $s_0(A,B) = s_0(B,A) = 0$. On the other hand, for $n\in\N$ usually we have $s_n(A,B)\neq s_n(B,A)$. For example, if $B\subset A$ then $s_n(B,A) = 0$, while $s_n(A,B)$ might be positive. Therefore, we also introduce the number
$$
s_n[A,B] := \max\{s_n(A,B),s_n(B,A)\}.
$$
The next proposition shows that this number is bounded by $n$.

%\marginpar{Difficult to read}
\begin{prop}\label{Thequantitys_n}
Let $A$ and $B$ be linear relations in $X$ such that $B$ is a one-dimensional perturbation of $A$. Then for $n\in\N\cup\{0\}$ we have
$$
s_n[A,B]\,\le\, n.
$$
\end{prop}

\begin{proof}
The claim is clear for $n=0$. Let $n\ge 1$. It obviously suffices to prove that $s_n(A,B)\le n$. If $A\subset B$ then $s_n(A,B) = 0$ and the desired inequality holds. Hence, let us assume that $\dim(A/A\cap B) = 1$ and set $C := A\cap B$.
%For this, we set $C := A\cap B$. If $C = A$, then $s_n(A,B) = 0$ and nothing is to show. Hence, let us assume that $\dim(A/C) = 1$.

Let $\calL$ be a subspace of $N(C)\cap R(C^n)$ such that $\calL\cap M(C^n) = \{0\}$. Towards a contradiction, suppose that $\dim(\calL\cap M(A^n)) > n$. So, there exist linearly independent vectors $x_{1,0},\ldots,x_{n+1,0}\in\calL\cap M(A^n)$. Then there exist $n+1$ singular chains of $A$ of the form
\[
X_k=(0,x_{k,n-1},\ldots,x_{k,0}), \quad k=1,\ldots,n+1,
\]
and $\{X_1,\ldots,X_{n+1}\}$ is linearly independent in $\calS_n^A$, c.f.\ \eqref{chains}.

By Lemma \ref{l:villa_crespo}, $\dim(\calS_n^A / \calS_n^C)\leq n$. Thus, there exists a non-trivial  $Y\in\calS_n^C$ such that $Y\in\linspan\{X_1,\ldots, X_{n+1}\}$, i.e.\ there exist $\alpha_1,\ldots,\alpha_{n+1}\in\mathbb{K}$ (not all zero) such that $Y=\sum_{k=1}^{n+1} \alpha_k X_k$.

So, $Y$ is a non-trivial singular chain of $C$ of the form $Y=(0,y_{n-1},\ldots,y_0)$, where
\[
y_j=\sum_{k=1}^{n+1} \alpha_k x_{k,j}, \quad j=0,1,\ldots,n-1.
\]
In particular, $y_0=\sum_{k=1}^{n+1} \alpha_k x_{k,0}\neq 0$ because $\{x_{1,0},\ldots,x_{n+1,0}\}$ is linearly independent. Now, since $x_{1,0},\ldots,x_{n+1,0}\in\calL$, also $y_0\in\calL$ and hence $y_0\in\calL\cap M(C^n)$, which is the desired contradiction.
\end{proof}

We now present our first generalization of Theorem 2.2 in \cite{blmt}.
In this case we assume that one of the two relations is a one-dimensional extension of the other.

\begin{thm}\label{t:Harburg}
Let $A$ and $B$ be linear relations in $X$ such that $A\subset B$ and $\dim(B/A) = 1$ and let $n\in\N\cup\{0\}$. Then the following holds:
\begin{enumerate}
	\item[(i)] $N(A^{n+1})/N(A^n)$ is finite-dimensional if and only if $N(B^{n+1})/N(B^n)$ is finite-dimensional. Moreover,
\begin{equation*}%\label{abschaetzung3}
-s_{n}(B,A)\,\le\,\dim\frac{N(B^{n+1})}{N(B^n)} -\dim\frac{N(A^{n+1})}{N(A^n)}\,\le\,1.
\end{equation*}
In particular, for $n\ge 1$ we have
\begin{equation}\label{abschaetzung4}
\left|\dim\frac{N(B^{n+1})}{N(B^n)} - \dim\frac{N(A^{n+1})}{N(A^n)}\right|\,\le\,\max\{1,s_n(B,A)\}\,\le\,n.
\end{equation}
	\item[(ii)] $N(A^n)$ is finite-dimensional if and only if $N(B^n)$ is finite-dimensional. Moreover,
for $n\geq 1$,
$$
\left|\dim N(B^n) - \dim N(A^n)\right|\,\le\,\sum_{k=0}^{n-1} \max\left\{1, s_k(B,A)\right\}\,\le\,\frac{(n-1) n}{2} +1.
$$
\end{enumerate}
\end{thm}
\begin{proof}
To prove the lower bound in item (i), suppose that there are
$$
m: = \dim \frac{N(B^{n+1})}{N(B^n)} + s_n(B,A) + 1
$$
linearly independent vectors $[x_{1,n}],\ldots,[x_{m,n}]$ in $N(A^{n+1})/N(A^n)$ and consider corresponding Jordan chains $(x_{k,n},\ldots,x_{k,0})$ of length $n+1$ of $A$, $k=1,\ldots,m$. By Lemma \ref{l:Encuentro}, the vectors $x_{1,0},\ldots, x_{m,0}$ are linearly independent and, if $\calL_0 := \linspan\{x_{1,0},\ldots,x_{m,0}\}$ then
\[
\calL_0\cap M(A^n) = \{0\}.
\]
Denote the cosets of the vectors $x_{k,n}$ in $N(B^{n+1})/N(B^n)$ by $[x_{k,n}]_B$, $k=1,\ldots,m$. Since
$$
s_n(B,A) = \max\left\{\dim(\calL\cap M(B^n)) : \calL\subset N(A)\cap R(A^n)\text{ subspace},\,\calL\cap M(A^n) = \{0\}\right\},
$$
Lemma \ref{l:Encuentro} implies that
\begin{align*}
\dim\linspan\{[x_{1,n}]_B,\ldots,[x_{m,n}]_B\}
&= m - \dim(\calL_0\cap M(B^n))\\
&\ge m - s_n(B,A) = \dim \frac{N(B^{n+1})}{N(B^n)} + 1,
\end{align*}
which is a contradiction.

On the other hand, assume that there are
\[
p:= \dim \frac{N(A^{n+1})}{N(A^n)} +2
\]
linearly independent vectors $[y_{1,n}]_B,\ldots, [y_{p,n}]_B$ in $\frac{N(B^{n+1})}{N(B^n)}$ and consider corresponding Jordan chains $(y_{k,n}, \ldots, y_{k,0})$ of length $n+1$ of $B$, for $k=1,\ldots,p$. By Lemma \ref{l:Encuentro}, the vectors $y_{1,0},\ldots, y_{p,0}$ are linearly independent and, if $\calL_Y := \linspan\{y_{1,0},\ldots,y_{p,0}\}$, then
\[
\calL_Y\cap M(B^n) = \{0\}.
\]
Now, applying Lemma \ref{l:shifting}, we obtain $p-1$ Jordan chains $(z_{k,n}, \ldots, z_{k,0})$
of length $n+1$ of $A$, $k=1,\ldots,p-1$, such that (after a possible reordering)
\[
z_{k,j}\in y_{k,j} + \linspan\{y_{p,l}: l=0,\ldots,j\} \quad \text{for} \ k=1,\ldots,p-1,\ j=0,\ldots,n.
\]
In particular, for each $k=1,\ldots,p-1$ there exists $\alpha_k\in\mathbb{K}$ such that $z_{k,0}=y_{k,0}+ \alpha_k y_{p,0}$.

Hence, if $\calL_Z:=\linspan\{z_{1,0},\ldots, z_{p-1,0}\}$ it is easy to see that
\[
\calL_Z\cap M(A^n)=\{0\},
\]
because $\calL_Z\subseteq\calL_Y$, $M(A^n)\subseteq M(B^n)$ and $\calL_Y\cap M(B^n)=\{0\}$. Thus, by Lemma \ref{l:Encuentro},
\begin{align*}
\dim \linspan\{[z_{1,n}],\ldots,[z_{p-1,n}]\}= \dim \calL_Z=p-1=\dim \frac{N(A^{n+1})}{N(A^n)} + 1,
\end{align*}
which is a contradiction.

In order to prove item (ii), note that for a linear relation $T$ we have
\begin{equation*}%\label{e:pieces}
N(T^n)= N(T)\oplus W_1\oplus\dots\oplus W_{n-1},
\end{equation*}
where $W_j$ is a subspace of $N(T^n)$ isomorphic to $\frac{N(T^{j+1})}{N(T^{j})}$ for $j=1,\ldots, n-1$. This fact follows easily by induction on $n$.
Hence, from item (i) we infer that $\dim N(A^n) < \infty$ if and only if $\dim N(B^n) < \infty$. % and that in this case we have
Also, as a consequence of \eqref{abschaetzung4} and Proposition \ref{Thequantitys_n},
\begin{align*}
\left|\dim N(B^n) - \dim N(A^n)\right|
&= \left|\sum_{k=0}^{n-1}\dim\frac{N(B^{k+1})}{N(B^k)} - \sum_{k=0}^{n-1}\dim\frac{N(A^{k+1})}{N(A^k)}\right|\\
&\le \sum_{k=0}^{n-1}\left| \dim\frac{N(B^{k+1})}{N(B^k)} - \dim\frac{N(A^{k+1})}{N(A^k)}\right|\\
&\le \sum_{k=0}^{n-1} \max\{1,s_k(B,A)\} = 1 + \sum_{k=1}^{n-1} k \\
& \le\, 1+\frac{(n-1) n}{2}.
\end{align*}
This concludes the proof of the theorem.
\end{proof}

The next theorem is the main result of this section. It states that the estimate obtained in \cite[Theorem 2.2]{blmt} for operators have to be adjusted when considering arbitrary linear relations. Note that $s_n[A,B] = 0$ for operators $A$ and $B$.

%It states that the upper bound in the estimate \eqref{e:ops} for operators has to be adjusted to $1 + s_n[A,B]$ when considering arbitrary linear relations. Note that $s_n[A,B] = 0$ for operators $A$ and $B$.

\begin{thm}\label{t:LaLucila}
Let $A$ and $B$ be linear relations in $X$ such that $B$ is a one-dimensional perturbation of $A$ and $n\in\N\cup\{0\}$. Then the following hold:
\begin{enumerate}
	\item[(i)]  $N(A^{n+1})/N(A^n)$ is finite-dimensional if and only if $N(B^{n+1})/N(B^n)$ is finite-dimensional. Moreover,
\begin{equation*}%\label{abschaetzung1}
-1-s_{n}(B,A)\,\le\,\dim\frac{N(B^{n+1})}{N(B^n)} - \dim\frac{N(A^{n+1})}{N(A^n)}\,\le\,1 + s_n(A,B).
\end{equation*}
In particular,
\begin{equation}\label{abschaetzung2}
\left|\dim\frac{N(B^{n+1})}{N(B^n)} - \dim\frac{N(A^{n+1})}{N(A^n)}\right|\,\le\,1 + s_n[A,B]\,\le\,n+1.
\end{equation}
\item[(ii)] $N(A^n)$ is finite-dimensional if and only if $N(B^n)$ is finite-dimensional. Moreover,
$$
\left|\dim N(B^n) - \dim N(A^n)\right|\,\le\,n + \sum_{k=0}^{n-1}s_k[A,B]\,\le\,\frac{n(n+1)}{2}.
$$
\end{enumerate}
\end{thm}
\begin{proof}
Define $C := A\cap B$. Then $C\subset A$ and $C\subset B$ as well as $\dim(A/C)\le 1$ and $\dim(B/C)\le 1$. Moreover, note that
$$
s_n(A,B) = s_n(A,C)
\quad\text{and}\quad
s_n(B,A) = s_n(B,C).
$$
Therefore, using the notation $D_n(T) = \dim\tfrac{N(T^{n+1})}{N(T^n)}$ for a relation $T$ in $X$, from Theorem \ref{t:Harburg} we obtain
$$
D_n(B) - D_n(A) = (D_n(B)-D_n(C)) - (D_n(A) - D_n(C))\,\le\,1 + s_n(A,B)
$$
Exchanging the roles of $A$ and $B$ leads to $D_n(A)-D_n(B)\le 1+s_n(B,A)$. This proves (i).

The proof of statement (ii) is analogous to the proof of its counterpart in Theorem~\ref{t:Harburg}. In this case, as a consequence of \eqref{abschaetzung2},
\begin{align*}
\left|\dim N(B^n) - \dim N(A^n)\right|
%&= \left|\sum_{k=0}^{n-1}\dim\frac{N(A^{k+1})}{N(A^k)} - \sum_{k=0}^{n-1}\dim\frac{N(B^{k+1})}{N(B^k)}\right|\\
\le \sum_{k=0}^{n-1}\left|D_k(A) - D_k(B)\right|\,\le\, \sum_{k=0}^{n-1}(1+s_k[A,B])\,\le\,\frac{n(n+1)}{2},
\end{align*}
and the theorem is proved.
\end{proof}

In Section \ref{s:sharpness} below we prove that the bound $n+1$ in \eqref{abschaetzung2} of Theorem \ref{t:LaLucila} is in fact sharp, meaning that there are examples of linear relations $A$ and $B$ which are one-dimensional perturbations of each other where the quantity on the left hand side of \eqref{abschaetzung2} coincides with $n+1$.

\medskip

The following corollary deals with linear relations without singular chains. If neither $A$ nor $B$ has singular chains then we recover the bounds from the operator case, see Theorem 2.2 in \cite{blmt}. %\ref{t:operators} in the Introduction.

\begin{cor}\label{c:no_sings_p=1}
Let $A$ and $B$ be linear relations in $X$ without singular chains such that $B$ is a one-dimensional perturbation of $A$. Then the following statements hold:
\begin{enumerate}
\item[\rm (i)]  $N(A^{n+1})/N(A^n)$ is finite dimensional if and only if $N(B^{n+1})/N(B^n)$ is finite dimensional. Moreover,
$$
\left|\dim\frac{N(A^{n+1})}{N(A^n)} - \dim\frac{N(B^{n+1})}{N(B^n)}\right|\,\le\,1.
$$
\item[\rm (ii)] $N(A^n)$ is finite dimensional if and only if $N(B^n)$ is finite dimensional. Moreover,
\begin{align*}
\left|\dim N (A^n) - \dim N(B^n)\right|\,\le\,n.
\end{align*}
\item[\rm (iii)] $N(A)\cap R(A^n)$ is finite dimensional if and only if $N(B)\cap R(B^n)$ is finite dimensional. Moreover,
$$
\left|\dim(N(A)\cap R(A^n)) - \dim(N(B)\cap R(B^n))\right|\,\le\,1.
$$
\end{enumerate}
\end{cor}

\begin{proof}
If $A$ and $B$ are linear relations in $X$ without singular chains, then $s_n[A,B]=0$ for each $n\in\N$. Therefore, items (i) and (ii) follow directly from items (i) and (ii) in Theorem \ref{t:LaLucila}. Finally, recall that for a linear relation $T$ in $X$ without singular chains we have $N(T^{n+1})/N(T^n)\cong N(T)\cap R(T^n)$, c.f.\  \cite[Lemma 4.4]{ssw}. Hence, (iii) follows from (i).
\end{proof}

\section[\hspace*{1cm}Sharpness of the bound in Theorem \ref{t:LaLucila}]{Sharpness of the bound in Theorem \ref{t:LaLucila}}\label{s:sharpness}
In this section we present an example which shows that the bound $n+1$ in Theorem \ref{t:LaLucila} can indeed be achieved and is therefore sharp. This is easy to see in the cases $n=0$ and $n=1$.

\begin{ex}\label{ex:LaPlata}
(a) Let $n=2$, and let $x_0,x_1,x_2,z_0,z_1,z_2,y_1,y_2,y_3$ be linearly independent vectors in $X$. Define the linear relations
\begin{align*}
A = \linspan\big\{&\{x_{2},x_{1}\},\{x_{1},x_{0}\},\{x_{0},0\},\\
&\{z_{2},z_{1}\},\{z_{1},z_{0}\},\{z_{0},0\},\\
&\boldsymbol{\{}\boldsymbol{y_3}\boldsymbol{,}\boldsymbol{x_{2}}\boldsymbol{-}\boldsymbol{y_2}
\boldsymbol{\}},\{x_{2}-y_2,y_1\},\{y_1,0\},\\
&\{z_{2},y_2\}\big\}
\end{align*}
and
\begin{align*}
B = \linspan\big\{&\{x_{2},x_{1}\},\{x_{1},x_{0}\},\{x_{0},0\},\\
&\{z_{2},z_{1}\},\{z_{1},z_{0}\},\{z_{0},0\},\\
&\{x_{2}-y_2,y_1\},\{y_1,0\},\\
&\{z_{2},y_2\},\boldsymbol{\{}\boldsymbol{y_2}\boldsymbol{,}\boldsymbol{0}\boldsymbol{\}}\big\}.
\end{align*}
All pairs are contained in both $A$ and $B$ except for the two pairs $\boldsymbol{\{}\boldsymbol{y_3}\boldsymbol{,}\boldsymbol{x_{2}}\boldsymbol{-}\boldsymbol{y_2}
\boldsymbol{\}}$ and $\boldsymbol{\{}\boldsymbol{y_2}\boldsymbol{,}\boldsymbol{0}\boldsymbol{\}}$
which are printed here in bold face.
Therefore, $A$ and $B$ are one-dimensional perturbations of each other.
It is easy to see that $M(A^2) = \linspan\{y_2-z_1, x_1-y_1-z_0\}$ and thus $M(A^2)\cap\linspan\{x_0,z_0,y_1\}=\{0\}$.
%Therefore, if a linear combination $\alpha x_{1,0}+\beta x_{2,0} + \gamma y_1$ is contained in $M(A^2)$, then there exists some $\delta\in\K$ such that $\{\delta(x_2,1}-y_2),\alpha x_{1,0}+\beta x_{2,0} + \gamma y_1\}\in A$. Now, using that $\{x_{2,1},x_{2,0}\},\{y_2,x_{1,1}-y_1\}\in A$, and thus $\{x_{2,1}-y_2,x_{2,0}-x_{1,1}+y_1\}\in A$, we obtain that $\alpha x_{1,0} + \beta x_{2,0} + \gamma y_1 - \delta(x_{2,0}-x_{1,1}+y_1)\in M(A) = \linspan\{x_{2,1}-y_2\}$. This immediately leads to $\alpha = \beta = \gamma = 0$. Thus,
By Lemma \ref{l:JC_li}, it follows that $[x_{2}]_A,[z_{2}]_A,[y_3]_A$ are linearly independent in $N(A^3)/N(A^2)$. As $N(B^2)=\linspan\{x_0,x_1,x_2,z_0,z_1,z_2,y_1,y_2\}$ it is clear that $N(B^3) = N(B^2)$, hence
$$
\dim\frac{N(A^3)}{N(A^2)} - \dim\frac{N(B^3)}{N(B^2)} = 3 - 0 = 3 = n+1.
$$
(b) Let $n\in \mathbb N$, $n>2$.  For our example we need $(n+1)^2$ linearly independent vectors in the linear space $X$, say $x_{i,j}$ for $i=1,\ldots,n$ and $j=0,\ldots,n$ as well as $y_1,\ldots,y_{n+1}$. Let us consider the linear relation
\begin{align*}
A =
&\linspan\left[\left\{\{x_{k,n},x_{k,n-1}\},\ldots,\{x_{k,1},x_{k,0}\},\{x_{k,0},0\} : k=1,\ldots,n\right\}\,\cup\right.\\
&\cup\, \{y_{n+1},x_{1,n}-y_n\}\,\cup\,\big\{\{x_{k,n}-y_{n-k+1},x_{k+1,n}-y_{n-k}\} : k=1,\ldots,n-2\big\}\\
& \left.\cup\,\{x_{n-1, n}-y_2,y_1\}\cup\,\{y_1,0\}\cup \{x_{n,n},y_n\}\cup \big\{\{y_l,y_{l-1}\}: l=3,\ldots,n\big\} \right].
\end{align*}
Notice that
\[
N(A)=\linspan\{x_{1,0},\ldots,x_{n,0}, y_1\}.
\]
In the following we compute the multivalued part of $A^k$ for $k=1,\ldots, n$. Assume that $x\in M(A)\subset R(A)$. Then
$\{0, x\}\in A$ and there exist scalars $\alpha_{i,j}, \beta_k, \gamma_l\in\K$ such that
\[
x=\sum_{i=1}^n\sum_{j=1}^n \alpha_{i,j}x_{i,j-1} + \sum_{k=1}^{n-2}\gamma_k(x_{k+1,n}-y_{n-k}) + \gamma_{n-1} y_1 + \gamma_n y_n + \beta_n (x_{1,n}-y_n) + \sum_{l=2}^{n-1}\beta_l y_l
\]
and
\begin{align*}
0
&= \sum_{i=1}^n\sum_{j=1}^n \alpha_{i,j}x_{i,j}+ \sum_{k=1}^{n-2}\gamma_k(x_{k,n}-y_{n-k+1}) + \gamma_{n-1}(x_{n-1,n}-y_2) + \gamma_n x_{n,n} + \sum_{l=2}^{n}\beta_l y_{l+1} \\
&= \sum_{i=1}^n (\alpha_{i,n} + \gamma_i)x_{i,n} + \sum_{i=1}^n\sum_{j=1}^{n-1} \alpha_{i,j}x_{i,j} + \beta_n y_{n+1} + \sum_{k=1}^{n-2}(\beta_{n-k}-\gamma_k) y_{n-k+1} - \gamma_{n-1}y_2.
\end{align*}
Therefore,
\[
\left\{
\begin{array}{rl}
\alpha_{i,n} + \gamma_i=0 & \text{for}\ i=1,\ldots,n,\\
\alpha_{i,j}=0& \text{for}\ i=1,\ldots,n, \ j=1,\ldots, n-1, \\
\beta_n=0 & , \\
\gamma_k-\beta_{n-k}=0 & \text{for}\ k=1,\ldots, n-2,\\
\gamma_{n-1}=0 & .
\end{array}
\right.
\]
Hence, we can rewrite the vector $x$ as
\begin{align*}
x=& \sum_{i=1}^{n-2}\alpha_{i,n}x_{i,n-1} + \alpha_{n,n}x_{n,n-1} +  \sum_{k=1}^{n-2}\gamma_k(x_{k+1,n}-y_{n-k}) + \gamma_n y_n + \sum_{l=2}^{n-1}\beta_l y_l \\ &=\sum_{k=1}^{n-2}\gamma_k (x_{k+1,n}-x_{k,n-1}) +  \gamma_n (y_n-x_{n,n-1}).
\end{align*}
Thus,
\begin{equation*}%\label{mul ex}
M(A)=\linspan\left(\{y_n-x_{n,n-1}\}\cup\big\{x_{k+1,n}-x_{k,n-1}: \ k=1,\ldots, n-2\big\}\right).
\end{equation*}
%Note that the vectors spanning $M(A)$ in \eqref{mul ex} are linearly independent. Therefore $\dim M(A)=n-1$.
If $x\in M(A^2)$, then there exists $y\in M(A)$ such that $\{y,x\}\in A$. Hence, if $y=\sum_{k=1}^{n-2}\alpha_k (x_{k+1,n}-x_{k,n-1}) + \alpha_{n-1}(y_n-x_{n,n-1})$ then
\begin{align*}
x-\sum_{k=1}^{n-2}\alpha_k (x_{k+1,n-1} - x_{k,n-2}) -\alpha_{n-1}(y_{n-1}-x_{n,n-2}) \in M(A).
\end{align*}
Therefore,
\begin{align*}
M(A^2)=\linspan& \left( \{y_n-x_{n,n-1}\}\cup\big\{x_{k+1,n}-x_{k,n-1}: \ k=1,\ldots, n-2\big\}\cup \right.\\
 & \cup\left. \{y_{n-1}-x_{n,n-2}\}\cup\big\{x_{k+1,n-1}-x_{k,n-2}: \ k=1,\ldots, n-2\big\}\right).
\end{align*}

Following the same arguments it can be shown that
\begin{align*}
M(A^{n-1})=\linspan& \left( \big\{x_{k+1,n-j}-x_{k,n-j-1}: \ k=1,\ldots, n-2,\ j=0,\ldots, n-2\big\}\cup \right.\\
 & \cup\left. \{y_{n-j}-x_{n,n-j-1}: \ j=0,\ldots, n-2\big\}\right).
\end{align*}
and
\begin{align*}
M(A^n)=\linspan& \left( \big\{x_{k+1,n-j}-x_{k,n-j-1}: \ k=1,\ldots, n-2,\ j=0,\ldots, n-1\big\}\cup \right.\\
 & \cup\left. \{y_{n-j}-x_{n,n-j-1}: \ j=0,\ldots, n-2\big\}\cup\{x_{n-1,n-1}-y_1-x_{n,0}\}\right).
\end{align*}
%where the last vector above is a consequence of $\{y_2-x_{n,1},x_{n-1,n-1}-y_1-x_{n,0}\}\in A$.
From this it follows that
\begin{equation}\label{e:eichstaett}
\linspan\{x_{1,0},\ldots,x_{n,0}, y_1\}\cap M(A^n)=\{0\}.
\end{equation}
Indeed, if $x$ is a vector contained in the set on the left hand side of \eqref{e:eichstaett}, then
$$
x = \alpha_1x_{1,0} + \dots + \alpha_nx_{n,0} + \alpha_{n+1}y_1 = \sum_{k=1}^{n-2}\beta_k(x_{k+1,1} - x_{k,0}) + \gamma(x_{n-1,n-1}-y_1-x_{n,0}),
$$
where $\alpha_j,\beta_k,\gamma\in\K$ for $j=1,\ldots,n+1$ and $k=1,\ldots,n-2$. This implies
$$
\sum_{k=1}^{n-2}(\alpha_k + \beta_k)x_{k,0} + \alpha_{n-1}x_{n-1,0} + (\alpha_n+\gamma)x_{n,0} + (\alpha_{n+1}+\gamma)y_1 - \sum_{k=1}^{n-2}\beta_kx_{k+1,1} - \gamma x_{n-1,n-1} = 0.
$$
Since all the vectors involved are by assumption linearly independent, it follows that $\gamma = 0$ and also $\beta_k = 0$ for $k=1,\ldots,n-2$ and thus also $\alpha_j = 0$ for all $j=1,\ldots,n+1$. That is, $x = 0$.

Now, it follows from \eqref{e:eichstaett} and Lemma~\ref{l:JC_li} that $[x_{1,n}]_A,\ldots,[x_{n,n}]_A, [y_{n+1}]_A$ are linearly independent in $N(A^{n+1})/N(A^n)$. On the other hand, if we consider the linear relation
\begin{align*}
B &=\linspan\left(\big\{ \{x_{k,j}, x_{k,j-1}\} : k=1,\ldots,n,\,j=1,\ldots,n \big\}
\cup\, \big\{\{x_{k,0},0\} : k=1,\ldots,n \big\}\right.\\
&\cup\, \big\{\{x_{k,n}-y_{n-k+1},x_{k+1,n}-y_{n-k}\} : k=1,\ldots,n-2\big\}\cup\{x_{n-1, n}-y_2,y_1\}\cup\,\{y_1,0\}\\
& \left.\cup  \big\{ \{x_{n,n},y_n\}, \{y_n,y_{n-1}\},\ldots,\{y_3,y_2\}\cup\{y_2,0\}\big\} \right),
\end{align*}
$A$ and $B$ are one-dimensional perturbations of each other. Also, it is straightforward to verify that $D(B) = N(B^n)$. In particular, $N(B^{n+1}) = N(B^n)$ so that
$$
\dim\frac{N(A^{n+1})}{N(A^n)} - \dim\frac{N(B^{n+1})}{N(B^n)} = n+1 - 0  = n+1,
$$
which shows that the worst possible bound is indeed achieved in this example.
\end{ex}

\section[\hspace*{1cm}Finite-dimensional perturbations]{Finite-dimensional perturbations}
\label{Osterhase}

A linear relation $B$ is a finite dimensional perturbation of another linear relation $A$ if both differ by finitely many dimensions from their intersection. Following \cite{ABJT}, we formalize this idea as follows.

\begin{defn}\label{laplata}
Let $A$ and $B$ be linear relations in $X$ and $p\in\N$. Then $B$ is called a {\em $p$-dimensional perturbation} of $A$ (and vice versa) if
\begin{equation*}
\max\left\{\dim\frac{A}{A\cap B},\,\dim\frac B{A\cap B}\right\} = p.
\end{equation*}
\end{defn}

%If $X$ is a Hilbert space and $A,B$ are closed linear relations in $X$ then
%both quantities $\dim \frac{A}{A\cap B}$ and $\dim \frac B{A\cap B}$ are finite
%if and only if $P_A-P_B$ is a finite rank operator, where
%$P_A$ and $P_B$ are the orthogonal projections  onto $A$ and $B$, respectively,
%cf.\ \cite{ABJT}.

\begin{rem}\label{r:sequence}
Let $A$ and $B$ be linear relations in $X$ which are $p$-dimensional perturbations of each other, $p>1$. Then it is possible to construct a sequence of one-dimensional perturbations, starting in $A$ and ending in $B$. Indeed, choose $\{\widehat f_1,\ldots, \widehat f_p\}$ and $\{\widehat g_1,\ldots,\widehat g_p\}$ in $X\times X$ such that
\begin{equation*}\label{Leipschg}
A = (A\cap B)\ds\linspan\{\widehat f_1,\ldots,\widehat f_{p}\}\quad
\mbox{and} \quad
B= (A\cap B)\ds\linspan\{\widehat g_1,\ldots,\widehat g_{p}\}.
\end{equation*}
Observe that $\{\widehat f_1,\ldots,\widehat f_p\}$ is
linearly independent if and
only if $ \dim\frac{A}{A\cap B}=p$. Otherwise, some of the elements
of $\{\widehat f_1,\ldots,\widehat f_p\}$
can be chosen as zero. An analogous statement
holds for $\{\widehat g_1,\ldots,\widehat g_p\}$.
Define $C_0:=A$ , $C_p:=B$, and
\begin{equation*}\label{e:Cs}
C_k:= (A\cap B)\ds
\linspan\{\widehat f_1,\ldots,\widehat f_{p-k},
\widehat g_{p-k+1},\ldots,\widehat g_p\},\quad k=1,\ldots,p-1.
\end{equation*}
Obviously, $C_{k+1}$ is a one-dimensional perturbation of $C_k$,
$k=0,\ldots,p-1$. If, in addition, $A\subset B$ is satisfied, then $\widehat f_{j}=0$ for all $j=1,\ldots,p$ holds and we obtain
$$
A\subset C_j \subset C_{j+1} \subset B \quad \mbox{for } j=1,\ldots,p-1.
$$
\end{rem}

\begin{thm}\label{t:pdim}
Let $A$ and $B$ be linear relations in $X$ such that $B$ is a $p$-dimensional perturbation of $A$, $p\ge 1$, and $n\in\N\cup\{0\}$. Then the following conditions hold:
\begin{enumerate}
\item[\rm (i)]  $N(A^{n+1})/N(A^n)$ is finite-dimensional if and only if $N(B^{n+1})/N(B^n)$ is finite-dimensional.  Moreover,
$$
\left|\dim\frac{N(A^{n+1})}{N(A^n)} - \dim\frac{N(B^{n+1})}{N(B^n)}\right|\,\le\,(n+1)p.
$$
\item[\rm (ii)]  If, in addition in item {\rm (i)}, $A\subset B$ is satisfied, then we have for $n\geq 1$
$$
\left|\dim\frac{N(A^{n+1})}{N(A^n)} - \dim\frac{N(B^{n+1})}{N(B^n)}\right|\,\le\,np.
$$
\item[\rm (iii)] $N(A^n)$ is finite-dimensional if and only if $N(B^n)$ is finite-dimensional.  Moreover,
$$
\left|\dim N(A^n) - \dim N(B^n)\right|\,\le\,\frac{n(n+1)}{2}p.
$$
\item[\rm (iv)]  If, in addition in item {\rm (iii)}, $A\subset B$ is satisfied, then we have for $n\geq 1$
    $$
\left|\dim N(A^n) - \dim N(B^n)\right|\,\le\,\frac{n(n-1)}{2}p+p.
$$
%\item[\rm (iii)] $N(A)\cap R(A^n)$ is finite-dimensional if and only if $N(B)\cap R(B^n)$ is finite-dimensio\-nal.  Moreover,
%$$
%\left|\dim(N(A)\cap R(A^n)) - \dim(N(B)\cap R(B^n))\right|\,\le\,p.
%$$
\end{enumerate}
\end{thm}

\begin{proof}
By Remark \ref{r:sequence} there exist linear relations $C_0,\ldots,C_p$ in $X$ with $C_0 = A$ and $C_p = B$ such that $C_{k+1}$ is a one-dimensional perturbation of $C_k$, $k=0,\ldots,p-1$. Hence, applying item (i) in Theorem \ref{t:LaLucila} repeatedly, we obtain
\begin{align*}
\left|\dim\frac{N(B^{n+1})}{N(B^n)} - \dim\frac{N(A^{n+1})}{N(A^n)}\right|
&\le\sum_{k=0}^{p-1}\left|\dim\frac{N(C_{k+1}^{n+1})}{N(C_{k+1}^n)} - \dim\frac{N(C_k^{n+1})}{N(C_k^n)}\right|\,\le\,(n+1)p.
\end{align*}
Also, applying item (ii) in Theorem \ref{t:LaLucila} repeatedly,
$$
\left|\dim N(A^n) - \dim N(B^n)\right|
\le\sum_{k=0}^{p-1}\left|\dim N(C_{k+1}^n) - \dim N(C_k^n)\right|
\,\le\,\frac{n(n+1)}{2}p,
$$
%Finally, for a linear relation $T$ in $X$ set $V_n(T) := N(T)\cap R(T^n)$. Applying item (ii) in Theorem \ref{t:LaLucila} repeatedly, we obtain
%$$
%\left|\dim(N(A)\cap R(A^n)) - \dim(N(B)\cap R(B^n))\right|
%\le\sum_{k=0}^{p-1}\left|\dim V_n(C_{k+1}) - \dim V_n(C_k)\right|\,\le\,p,
%$$
which shows (iii). The statements (ii) and (iv) in the case $A\subset B$ follows in the same way
from Remark \ref{r:sequence} and Theorem \ref{t:Harburg}.
\end{proof}

For linear relations $A$ and $B$ without singular chains we obtain the same (sharp) estimates as for operators, see \cite{blmt}.

\begin{cor}\label{c:no_sings}
Let $A$ and $B$ be linear relations in $X$ without singular chains such that $B$ is a $p$-dimensional perturbation of $A$, $p\ge 1$. Then the following conditions hold:
\begin{enumerate}
\item[\rm (i)]  $N(A^{n+1})/N(A^n)$ is finite-dimensional if and only if $N(B^{n+1})/N(B^n)$ is finite-dimensional. Moreover,
$$
\left|\dim\frac{N(A^{n+1})}{N(A^n)} - \dim\frac{N(B^{n+1})}{N(B^n)}\right|\,\le\,p.
$$
\item[\rm (ii)] $N(A^n)$ is finite-dimensional if and only if $N(B^n)$ is finite-dimensional. Moreover,
\begin{align*}
\left|\dim N (A^n) - \dim N(B^n)\right|\,\le\,np.
\end{align*}
\item[\rm (iii)] $N(A)\cap R(A^n)$ is finite-dimensional if and only if $N(B)\cap R(B^n)$ is finite-dimensio\-nal.  Moreover,
$$
\left|\dim(N(A)\cap R(A^n)) - \dim(N(B)\cap R(B^n))\right|\,\le\,p.
$$
\end{enumerate}
\end{cor}

\begin{proof}
The claims follow immediately applying repeatedly the results in Corollary \ref{c:no_sings_p=1} to the finite sequence of one-dimensional prturbations $A=C_0, C_1,\ldots, C_p=B$.
%By \cite[Lemma 4.4]{ssw}, for a linear relation $T$ in $X$ without singular chains we have $N(T^{n+1})/N(T^n)\cong N(T)\cap R(T^n)$. Hence, (i) follows from Theorem \ref{t:pdim}(iii). Moreover, (ii) follows from (i) by taking \eqref{e:pieces} into account.
\end{proof}

\vspace*{.3cm}
\section[\hspace*{1cm}Rank-one perturbations of matrix pencils]{Rank-one perturbations of matrix pencils}\label{Nicki}
In this section we  apply our results to matrix pencils $P$ of the form
\begin{equation*}%\label{pencil}
P(s):= sE-F,
\end{equation*}
where $s\in \mathbb C$ and  $E$, $F$ are square matrices in $\mathbb C^{d\times d}$.
We will estimate the change of the number of Jordan chains of $P$
under a perturbation with a rank-one matrix pencil.

We do not assume $E$  to be invertible. Nevertheless,
 if we identify $E$ with
the linear relation given by the graph of  $E$ then we have an inverse $E^{-1}$ of $E$
in the sense of linear relations, see \eqref{Dammtor}.
Also, we have that
\begin{equation*}%\label{Blankspace}
E^{-1}F = \left\{\{x,y\}\in \mathbb C^d\times \mathbb C^d : Fx=Ey \right\} = N [F \ -E].
\end{equation*}

Recall that $\lambda\in\mathbb{C}$ is an eigenvalue of $P(s)=sE-F$ if zero is an eigenvalue of $P(\lambda)$, and $\infty$ is an eigenvalue of $P$ if zero is an eigenvalue of the dual matrix pencil $(\rev P)(s)=sF-E$.
%We denote the set of all eigenvalues of the pencil $\mathcal A$ with $\sigma_p(\mathcal A)$.
In the following we recall the notion of Jordan chains for matrix pencils, see  e.g.\ \cite[Section 1.4]{GLR09}, \cite{K51}, or \cite[\S 11.2]{M88}.

\begin{defn}\label{LA}
An ordered set $(x_n, \ldots ,x_{0})$ in $\mathbb C^d$ is a \emph{Jordan chain of length $n+1$ at  $\lambda\in\ol{\C}:=\C\cup\{\infty\}$} (for the matrix pencil $P(s)$) if $x_{0}\neq 0$ and
\begin{equation*}%\label{eq:chain_pencil}
\begin{array}{lrrrr}
\lambda\in\mathbb C: & (F -\la E)x_0=0, & (F-\lambda E)x_1 = E x_0, & \ldots, & (F-\lambda E)x_n = E x_{n-1}, \\[2mm]
\lambda=\infty: & E x_0 = 0, & E x_1  = F x_0, & \ldots,& E x_n = F x_{n-1}.
 \end{array}
\end{equation*}
Moreover, we denote by $\mathcal L_{\lambda}^l(P)$ the subspace spanned by the vectors of all Jordan chains up to length $l\ge 1$ at $\lambda \in \overline{\mathbb C}$.
If $l=0$ or if $\la$ is not an eigenvalue of $P$ we define $\mathcal L_{\lambda}^l(P)=\{0\}$.
\end{defn}

\begin{rem}
As mentioned above, Definition~\ref{LA} is inspired by the
definition of eigenvalues of pencils introduced in 1951 by M.V.\ Keldysh
who used this concept in his study  of operator pencils, see \cite{K51,M88}. This definition fits to the definition of eigenvalues of linear relations and in that
way to the purpose of this paper.

The authors are aware of the fact  that in many recent publications in the matrix pencil community, see for instance \cite{BR20,DD07,DMT08,DMT08sing,MMW15,S83}, a different definition
for eigenvalues of matrix pencils is used which is based on changes in the rank of $P(s)$. What in our paper is called an eigenvalue is there sometimes called a singular point.
However, this concept is not used in the community of linear relations, so we apologize and warn for possible misunderstandings.
\end{rem}

Given a matrix pencil $P(s)$, the aim of this section is to obtain lower and upper bounds for the difference
\[
\dim\frac{\calL_\la^{n+1}(P +Q)}{\calL_\la^{n}(P+Q)} - \dim\frac{\calL_\la^{n+1}(P)}{\calL_\la^{n}(P)},
\]
where $Q$ is a rank-one matrix pencil, $n\in\N\cup\{0\}$ and $\lambda\in\overline{\mathbb{C}}$.

\medskip

We start with a simple lemma, which follows directly from the definitions. It allows us to reduce the study of Jordan chains at some $\la\in\ol\C$ to Jordan chains at zero.

\begin{lem}\label{Marseille0}
Given a matrix pencil $P(s)=sE-F$, the following statements hold:
%consider the matrix pencil $\mathcal{B}(s)=sE-(F-\lambda E)$. Then,
\begin{enumerate}
\item[{\rm (i)}]  $(x_n, \ldots ,x_{0})$ is a Jordan chain of $P$ at $\lambda\in\mathbb C$ if and only if it is a Jordan chain of the matrix pencil $\tilde{P}(s):=sE-(F-\lambda E)$ at zero.
%\item[{\rm (ii)}]  For $\lambda\in\mathbb C$ the set $(x_n, \ldots ,x_{0})$ is a quasi-Jordan chain of the linear relation $E^{-1}F$ at $\lambda$ if and only if it is a quasi-Jordan chain of the linear relation $E^{-1}F-\lambda$.
\item[{\rm (ii)}] $(x_n, \ldots ,x_{0})$ is a Jordan chain of $P(s)$ at  $\infty$ if and only if
it is a Jordan chain of the dual matrix pencil $(\rev P)(s):= sF- E$ at zero.
%\item[{\rm (iv)}]  The set $(x_n, \ldots ,x_{0})$ is a Jordan chain of the linear relation
%$E^{-1}F$ at $\infty$ if and only if it is a Jordan chain of the linear relation $F^{-1}E$.
\end{enumerate}
\end{lem}

The following proposition shows that the Jordan chains of the matrix pencil $P(s)$ coincide with the Jordan chains of the linear relation $E^{-1}F$. As the proof is simple and straightforward, we omit it.

\begin{prop}\label{FuerFritze}
For $n\in\N\cup\{0\}$ and $\la\in\ol\C$ the following two statements are equivalent.
\begin{itemize}
\item[\rm (i)]  $(x_n, \ldots ,x_{0})$ is a Jordan chain of $P$ at $\la$.
\item[\rm (ii)] $(x_n, \ldots ,x_{0})$ is a quasi-Jordan chain of $E^{-1}F$ at $\la$.
\end{itemize}
In particular, for $\la\in\C$ we have
$$
\calL_\la^n(P) = N((E^{-1}F - \la)^n).
$$
\end{prop}

Note that the quasi-Jordan chains of a linear relation $A$ at $\infty$ are the same as the quasi-Jordan chains of the inverse linear relation $A^{-1}$ at zero.
Moreover, it is easy to see that $E^{-1}F=(F^{-1}E)^{-1}$. Therefore,

\begin{cor}\label{infinito}
$(x_n, \ldots ,x_{0})$ is a Jordan chain of $P(s)=sE-F$ at $\infty$ if and only if $(x_n, \ldots ,x_{0})$ is a quasi-Jordan chain of $F^{-1}E$ at zero. In particular,
\[
\calL_\infty^n(P) = M((E^{-1}F)^n)=(N((F^{-1}E)^n)=\calL_0^n(\rev P).
\]
\end{cor}

Due to Proposition \ref{FuerFritze}, for $n\in\N\cup\{0\}$ and $\la\in\C$ we have
$$
\dim\frac{\calL_\la^{n+1}(P)}{\calL_\la^{n}(P)} = \dim\frac{N((E^{-1}F-\la)^{n+1})}{N((E^{-1}F-\la)^{n})}.
$$
On the other hand, Corollary \ref{infinito} implies that
\[
\dim\frac{\calL_\infty^{n+1}(P)}{\calL_\infty^{n}(P)} = \dim\frac{N((F^{-1}E)^{n+1})}{N((F^{-1}E)^{n})}.
\]

For a given a matrix pencil $P(s)=sE-F$ we now consider perturbations %for the matrix
%in \eqref{pencil}
of the form
\begin{equation}\label{stoerung}
 Q(s)=w(su^*-v^*), %,\quad w\neq 0.
\end{equation}
where $u,v,w\in\mathbb{C}^d$, $(u,v)\neq(0,0)$ and $w\neq 0$. These are rank-one matrix pencils. Recall that the rank of a matrix pencil $Q$ is the largest $r\in\mathbb N$ such that $Q$, viewed as a matrix with polynomial entries, has minors of size $r$  that are not identically zero \cite{DMT08,G59}.
Then, $P$ and $P+Q$ are rank-one perturbations of each other, in the sense that they differ by (at most) a rank-one matrix pencil.

\begin{lem}\label{l:kernelvsrange}
Given $P(s)=sE-F$, let $Q$ be a rank-one matrix pencil as in \eqref{stoerung}.
Then, the linear relations
$$
E^{-1}F \quad \mbox{and} \quad  \left(E+wu^*\right)^{-1}(F+wv^*)
$$
either coincide or they are one-dimensional perturbations of each other
in the sense of Definition \rmref{laplata}.
%If $\mathcal \mathcal P(s)=(su+v) w^*$ then the two linear relations
%$$
%FE^{-1} \quad \mbox{and} \quad
%(F+vw^*) \left( E+uw^*\right)^{-1}
%$$
%are one-dimensional perturbations of each other in the sense of Definition \rmref{laplata}.
\end{lem}
\begin{proof}
Obviously, for $\calM := E^{-1}F \cap \left(E+wu^*\right)^{-1}(F+wv^*)$ we have
$$
\mathcal M = \left\{
 %\left(
%\begin{array}{c}
%x\\
%y
%\end{array}
%\right)
\{x,y\} \in \mathbb C^d\times \mathbb C^d :
Fx=Ey\;
\mbox{ and }\; (F+wv^*)x=(E+wu^*)y
\right\}.
$$
That is,
$$
\mathcal M = E^{-1}F \cap \{v, -u\}^\bot
 %\left(\begin{array}{c}
%v\\
%-u
%\end{array}\right)^\perp
= \left(E+wu^*\right)^{-1}(F+wv^*) \cap \{v, -u\}^\bot.
 %\left(\begin{array}{c}
%v\\
%-u
%\end{array}\right)^\perp.
$$
This implies
$$
\dim \frac{E^{-1}F}{\mathcal M}\leq 1 \qquad \mbox{and}\qquad
\dim \frac{\left(E+wu^*\right)^{-1}(F+wv^*)}{\mathcal M}\leq 1,
$$
which proves the claim.
\end{proof}

Matrix pencils as in \eqref{stoerung} do not cover the set of all rank-one matrix pencils in $\C^d$. The remaining rank-one matrix pencils can be written as
\begin{equation}\label{otras}
	Q(s)=(su-v)w^*,
\end{equation}
where $u,v,w\in\C^d$ are such that $(u,v)\neq(0,0)$ and $w\neq 0$. Given $P(s)=sE-F$ and a rank-one pencil $Q$ of the form \eqref{otras}, the associated linear relations $E^{-1}F$ and $(E+uw^*)^{-1}(F+vw^*)$ can be two-dimensional perturbations of each other. Hence, the statements in Lemma \ref{l:kernelvsrange} are not valid for  rank-one matrix pencils of the form\eqref{otras}.

On the other hand, the linear relations $FE^{-1}$ and $(F+vw^*)(E+uw^*)^{-1}$ are (at most) one-dimensional perturbations of each other in the sense of Definition \ref{laplata}.  A deeper analysis of the correspondence between matrix pencils and their representing linear relations will be provided in the forthcoming manuscript \cite{GMPPT}, where the Segre and Weyr characteristics for linear relations are introduced. The results will then give rise to sharp estimates on similar quantities as above for (all) one-dimensional perturbations.

\begin{rem}\label{duales}
Applying Lemma \ref{l:kernelvsrange} to the dual matrix pencils $\rev P$ and $\rev Q$, it follows that
\[
F^{-1}E \ \ \ \text{and} \ \ \ (F + wv^*)^{-1}(E+wu^*)
\]
either coincide or they are one-dimensional perturbations of each other
in the sense of Definition \rmref{laplata}.
\end{rem}

The following theorem is the second main result of this article. We consider here
all possible situations of regular/singular matrix pencils $P$ and $P+Q$. Recall that
a matrix pencil $P(s)=sE-F$ is called \textit{regular} if $\det(sE-F)$ is not identically zero. Otherwise, $P$ is called  \textit{singular}.

\begin{thm}\label{grosserWurf}
Given $P(s)=sE-F$, let $Q$ be a rank-one matrix pencil as in \eqref{stoerung}. For  $\lambda \in \overline{\mathbb C}$ and  $n\in\N\cup\{0\}$, the following statements hold:
\begin{itemize}
\item[\rm (i)] If both pencils $P$ and $P+Q$ are
regular, then %we have for $k\in\N\setminus\{0\}$
	\begin{align*}
\left|\dim\frac{\mathcal{L}_{\lambda}^{n+1}(P+Q)}{\mathcal{L}_{\lambda}^n(P+Q)}-
\dim\frac{\mathcal{L}_{\lambda}^{n+1}(P)}{\mathcal{L}_{\lambda}^{n}(P)}\right|
\leq 1.
	\end{align*}
\item[\rm (ii)]
 If $P$ is regular but $P+Q$ is singular, then %we have for $k\in\N\setminus\{0\}$
	\begin{align*}
-1-n \leq  \dim\frac{\mathcal{L}_{\lambda}^{n+1}(P+Q)}{\mathcal{L}_{\lambda}^n(P+Q)}-
\dim\frac{\mathcal{L}_{\lambda}^{n+1}(P)}{\mathcal{L}_{\lambda}^{n}(P)}
\leq 1.
	\end{align*}
\item[\rm (iii)]
 If $P$ is singular and $P+Q$ is regular, then %we have for $k\in\N\setminus\{0\}$
	\begin{align*}
-1 \leq  \dim\frac{\mathcal{L}_{\lambda}^{n+1}(P+Q)}{\mathcal{L}_{\lambda}^n(P+Q)}-
\dim\frac{\mathcal{L}_{\lambda}^{n+1}(P)}{\mathcal{L}_{\lambda}^{n}(P)}
\leq n+1.
	\end{align*}
\item[\rm (iv)]
 If both $P$ and $P+Q$ are singular, then %we have for $k\in\N\setminus\{0\}$
	\begin{align*}
\left|\dim\frac{\mathcal{L}_{\lambda}^{n+1}(P+Q)}{\mathcal{L}_{\lambda}^n(P+Q)}-
\dim\frac{\mathcal{L}_{\lambda}^{n+1}(P)}{\mathcal{L}_{\lambda}^{n}(P)}\right|
\leq n+1.
	\end{align*}
\end{itemize}
\end{thm}

\begin{proof}
According to Lemma~\ref{Marseille0}, if $\la\in\C$ we may assume $\lambda = 0$. By
Proposition \ref{FuerFritze}, for $n\in\N\cup\{0\}$ we have that
\begin{align*}
\mathcal{L}_{0}^{n}(P) = N\big((E^{-1}F)^n\big)
\quad \mbox{and} \quad
\mathcal{L}_{0}^{n}(P+Q) = N(B^n),
\end{align*}
where $B:=(E+wu^*)^{-1}(F+wv^*)$.
Due to Lemma~\ref{l:kernelvsrange} the linear relations $E^{-1}F$ and $B$ are (at most) one-dimensional perturbations of each other and, by  Theorem~\ref{t:LaLucila},
\begin{equation*}
-1-s_{n}(B,E^{-1}F)\le\,\dim\frac{\mathcal{L}_{0}^{n+1}(P+Q)}{\mathcal{L}_{0}^n(P+Q)}-
\dim\frac{\mathcal{L}_{0}^{n+1}(P)}{\mathcal{L}_{0}^{n}(P)}\le 1 + s_n(E^{-1}F,B).
\end{equation*}
Then, Proposition~\ref{Thequantitys_n} implies statement (iv). If the pencil $P$
is regular then, by definition, not every complex number is an eigenvalue of $P$. Hence,
by  Proposition~\ref{FuerFritze}, those numbers are neither eigenvalues of $E^{-1}F$.
From \cite{ssw05} it follows that, in this case, $E^{-1}F$ has no singular chains
and we conclude that
$$
s_n(E^{-1}F,B)=0,
$$
see \eqref{Def_s_n}.
Similarly, if $P+Q$ is regular we obtain $s_{n}(B,E^{-1}F)=0$,
which shows the remaining statements (i)--(iii).

For $\la=\infty$ similar arguments can be used using $F^{-1}E$ and $C:=(F+wv^*)^{-1}(E+wu^*)$ instead of $E^{-1}F$ and $B$, see Corollary \ref{infinito} and Remark \ref{duales}.
\end{proof}

Note that the estimate in item (i) of Theorem \ref{grosserWurf} is already known. It was shown in \cite[Lemma 2.1]{DMT08}
with the help of a result for polynomials, see also \cite[Theorem 1]{T80}. The remaining estimates in Theorem \ref{grosserWurf} are completely new.

\begin{ex}
In this and the following section we focus on matrix pencils, but
most of the statements remain true if we consider operator pencils of the form
$$
Z(s) := sE-F,
$$
where $E$ and $F$ are linear and bounded operators in some Hilbert space $X$.
If $E$ and $F$ are compact operators, then $Z(s)$ is a Keldysh pencil, see \cite{K51}.

Assume that $E$ and $F$ are bounded operators.
One defines eigenvalues and Jordan chains as in Definition~\ref{LA} and it is easily seen
that also Lemma~\ref{Marseille0}, Proposition~\ref{FuerFritze}, Corollary~\ref{infinito}
and Lemma~\ref{l:kernelvsrange} hold, as they are based on algebraic properties only,
where $Q(s)$ for some vectors $u,v,w$ in $X$ with $(u,v)\neq (0,0)$ and $w\neq0$ is defined as
$$
Q(s)x = w(s\langle x,u\rangle-\langle x,v\rangle), \qquad x\in X.
$$
Here $\langle \cdot\,,\cdot\rangle$ stands for the Hilbert space scalar product
in $X$. Then a straight-forward application of  Theorem~\ref{t:LaLucila}
(see also Theorem \ref{grosserWurf}) gives for $\lambda \in \mathbb C$
\begin{equation}\label{lindo}
\left|\dim\frac{\mathcal{L}_{\lambda}^{n+1}(Z+Q)}{\mathcal{L}_{\lambda}^n(Z+Q)}-
\dim\frac{\mathcal{L}_{\lambda}^{n+1}(Z)}{\mathcal{L}_{\lambda}^{n}(Z)}\right|\le n+1,
\end{equation}
if $\frac{\mathcal{L}_{\lambda}^{n+1}(Z)}{\mathcal{L}_{\lambda}^{n}(Z)}$ is of finite dimension. The estimate in \eqref{lindo} seems to be new for operator pencils.
Moreover, in this setting also essential spectrum may exist. We are not going into
details here, but the above setting also allows to treat the essential spectrum. We refer
to \cite{GMPST20} for related considerations.
\end{ex}

\begin{rem}
In the following we present estimates for so-called {\it Wong sequences}, which have their origin in~\cite{Wong74}. Recently, Wong sequences have been
used to prove the Kronecker canonical form,  see~\cite{BergIlch12, BergTren12, BergTren13}. For $E,F\in\C^{d\times d}$ the Wong sequence of the second kind of the pencil $P(s) := sE-F$ is defined as the sequence of subspaces
%$(\mathcal V_i(\mathcal A))_{i\in\N}$ and
$(\mathcal W_i(P))_{i\in\N}$ given by
\begin{align*}
%\mathcal V_0(\mathcal A) & = \mathbb C^d,\quad &\mathcal V_{i+1}(\mathcal A) & =
%\left\{x\in\mathbb C^d : Fx\in E\mathcal V_i(\mathcal A)\right\},\quad i\in\mathbb N,\\
\mathcal W_0(P) & = \{0\},\quad &\mathcal W_{i+1}(P) & =
\left\{x\in\mathbb C^d : Ex\in F\mathcal W_i(P)\right\},\quad i\in\N\cup\{0\}.
\end{align*}
It is easily seen by induction that for $n\in\N$ we have
$$
%\calV_k(\calA) = D\big((E^{-1}F)^k\big)
%\qquad\text{and}\qquad
\calW_n(P) = %M\big((E^{-1}F)^k\big) =
N\big((F^{-1}E)^n\big).
$$
Theorem \ref{t:LaLucila} now yields the following statements on the behavior of the Wong sequences of the second kind under rank-one perturbations of the type \eqref{stoerung}:
\begin{itemize}
\item[\rm (i)] If both pencils $P$ and $P+Q$ are
regular, then %we have for $k\in\N\setminus\{0\}$
	\begin{align*}
	 \left|\dim\frac{\mathcal W_{n+1}(P+Q)}{\mathcal{W}_n(P+Q)}-
\dim\frac{\mathcal{W}_{n+1}(P)}{\mathcal{W}_{n}(P)}\right|
\leq 1.
	\end{align*}
\item[\rm (ii)]
 If $P$ is regular but $P+Q$ is singular, then
	\begin{align*}
-1-n \leq  \dim\frac{\mathcal{W}_{n+1}(P+Q)}{\mathcal{W}_n(P+Q)}-
\dim\frac{\mathcal{W}_{n+1}(P)}{\mathcal{W}_{n}(P)}
\leq 1.
	\end{align*}
\item[\rm (iii)]
 If $P$ is singular and $P+Q$ is regular, then
	\begin{align*}
-1 \leq  \dim\frac{\mathcal{W}_{n+1}(P+Q)}{\mathcal{W}_n(P+Q)}-
\dim\frac{\mathcal{W}_{n+1}(P)}{\mathcal{W}_{n}(P)}
\leq n+1.
	\end{align*}
\item[\rm (iv)]
 If both $P$ and $P+Q$ are singular, then
\begin{align*}
\left|\dim\frac{\mathcal{W}_{n+1}(P+Q)}{\mathcal{W}_n(P+Q)}-
\dim\frac{\mathcal{W}_{n+1}(P)}{\mathcal{W}_{n}(P)}\right|
\leq n+1.
\end{align*}
\end{itemize}
\end{rem}

\section{Perturbations of the Kronecker canonical form}\label{KCF}

%In the following we recall some basic definitions and results for matrix pencils and use the opportunity to fix some notation.
Recall that every pencil $P(s)=sE-F$ can be transformed into the \textit{Kronecker canonical form}, see e.g.~\cite{BergTren12,BergTren13,G59}. To introduce this form, define for
$k\in\N$ the matrices
\[
N_k:=\left[\begin{array}{cccc}
0&&&\\
1&0&&\\
&\ddots&\ddots&\\
&&1&0
\end{array}
\right]\in\C^{k\times k},
\]
and for a multi-index  $\alpha=(\alpha_1,\ldots,\alpha_l)\in\N^l$, $l\geq 1$, with absolute value $|\alpha|=\sum_{i=1}^l\alpha_i$  let
\[
N_{\alpha}:=\diag(N_{\alpha_1},\ldots,N_{\alpha_l})\in\C^{|\alpha|\times |\alpha|}.
\]

If $k\ge 1$, the following rectangular matrices are defined as
\[
K_k:=\left[\begin{array}{cccc}
1 & 0 &&\\
& \ddots& \ddots&\\
&&1&0
\end{array}
\right],\quad
L_k:=\left[\begin{array}{cccc}
0 & 1 &&\\
& \ddots& \ddots&\\
&&0&1
\end{array}
\right]\in\C^{k\times (k+1)},
\]
and, if $k=0$,
\[
    K_0 = L_0 := 0_{0\times 1}.
\]
If $E,F\in \C^{d\times d}$, the expression $0_{0\times 1}$ means that there is a $0$-column $(0,\ldots,0)^\top\in\C^{d\times 1}$ in the matrix~\eqref{kcf} below,
and $0_{0\times 1}^\top$ means that there is a $0$-row $(0,\ldots,0)\in\C^{1\times d}$ in \eqref{kcf} at the corresponding block.
The notation $0_{0\times 1}$ indicates that there is no contribution to the number of rows in \eqref{kcf}, whereas $0_{0\times 1}^\top$
gives no contribution to the number of columns.
For a multi-index $\veps=(\veps_1,\ldots,\veps_l)\in(\N\cup\{0\})^l$ we define
\begin{align*}
%\label{KalphaLalpha}
K_{\veps}:=\diag(K_{\veps_1},\ldots,K_{\veps_l}),~ L_{\veps}:=\diag(L_{\veps_1},\ldots,L_{\veps_l})\in\C^{|\veps|\times(|\veps|+l)}.
\end{align*}

According to Kronecker~\cite{K90},
there exist invertible matrices $S,T\in\C^{d\times d}$ such that $S(sE-F)T$ has a block diagonal form
\begin{align}
\label{kcf}
\begin{bmatrix} sI_{n_0}-A_0& 0&0&0 \\0& sN_\alpha-I_{|\alpha|}&0&0\\ 0&0& sK_{\veps}-L_{\veps} &0\\ 0&0&0& sK_{\eta}^\top-L_{\eta}^\top\end{bmatrix}
\end{align}
for some $A_0\in\C^{n_0\times n_0}$ in Jordan canonical form, which is unique up to a permutation of its Jordan blocks, and multi-indices $\alpha\in\N^{n_\alpha}$, $\veps\in(\N\cup\{0\})^{n_\veps}$, $\eta\in(\N\cup\{0\})^{n_{\eta}}$ which are unique up to a permutation of their entries, see also \cite[Chapter XII]{G59} or \cite{KM06}.
%The eigenvalues of the matrix $A_0$ are called the {\em proper eigenvalues values} of the pencil $P(s)$.
Let
 $(\sigma_1(\la),\ldots,\sigma_r(\la))$ denote the sizes of the Jordan blocks in a non-increasing order associated to an eigenvalue $\la$ of $A_0$.
These numbers are also called the {\em Segre characteristic} of the eigenvalue $\la$ of $A_0$.
The numbers
 $\alpha_i$, $i=1,\ldots,n_\alpha$ are called the {\em infinite elementary divisors of $P(s)$}, the numbers
 $\veps_i$, $i=1,\ldots,n_\veps$ are called the {\em column minimal indices of $P(s)$}, and the numbers $\eta_j$, $j=1,\dots,n_\eta$ are known as the {\em row minimal indices of $P(s)$}, see e.g.~\cite{DD07,G59}. It is assumed that they are indexed in non-increasing order,
 i.e.\
\begin{equation}\label{pehquh}
\alpha_1\geq\ldots\geq\alpha_{n_\alpha}\geq 1, \quad
\veps_1\geq\ldots\geq\veps_{n_\veps}\geq 0 \quad \text{and} \quad \eta_1\geq \ldots\geq \eta_{n_\eta}\geq 0.
\end{equation}
The sequences of numbers in \eqref{pehquh} are also called the  {\em Segre characteristics} of the infinite elementary divisors, the column minimal indices and the row minimal indices of the
pencil $P(s)$. Note that the Segre characteristic in \cite{DD07} was defined in a slightly different
way, namely without the numbers stemming from the minimal indices.

For $\la\in\C$ the {\em Weyr characteristic} of $A_0$ is defined for each $j\in\N$ as
\begin{equation}\label{WeyrWeyr}
w_j(\la)=\#\{i: \sigma_i(\la)\geq j\},~ j=1,\ldots,\sigma_1, \quad w_j(\la)=0, ~j>\sigma_1,
\end{equation}
i.e., $w_j(\la)$ is the number of Jordan blocks of size at least $j$ of the eigenvalue $\la$ of $A_0$ . If $\la$ is not an eigenvalue of $A_0$ we define  $w_j(\la)=0, ~j\in\N$. Note that
$$
w_j(\la)=\dim \frac{N((A_0-\la)^j)}{N((A_0-\la)^{j-1})}.
$$
In the same way, the {\em Weyr characteristics} of the infinite elementary divisors, the column minimal indices
and the row minimal indices are defined as the conjugate partitions of $\alpha$, of $\veps$, and of $\eta$.
E.g., if $\veps_1\geq\ldots\geq \veps_{n_\veps}\geq 0$ are the column minimal indices of $P(s)$, then
\begin{equation}\label{WeyrWeyrWeyr}
\Delta_j:=\#\{i:\ \veps_i\geq j\}, \qquad j=0,\ldots,\veps_{1}, \quad \Delta_j=0,~j>\veps_1,
\end{equation}
is the {\em Weyr characteristic} of the column minimal indices of $P(s)$
i.e.\ $\Delta_j$ is the number of column minimal indices of $P(s)$ which are larger than or equal to $j$.
The finite sequences $(\Delta_1, \ldots, \Delta_{\veps_1})$ and $(\veps_1,\ldots, \veps_{n_\veps})$ are conjugate partitions of $|\veps|$.
Note that the Segre characteristics can be easily derived from the Weyr characteristics. For a detailed exposition of the Weyr characteristic of matrices we refer to \cite{Sha99}.

If the Kronecker canonical form of $P$ is given by \eqref{kcf}, then
\[
\rank(P)=d-n_\veps=d-n_\eta,
\]
i.e.\ the rank of a pencil is related to the number of column and row minimal indices of $P$.
In what follows we will investigate the behavior of the
Kronecker canonical form under perturbations. In
 \cite{DD07} the unperturbed pencil  $P$ has no full rank, and the perturbation $Q$ is a pencil such that
\[
\rank(P+Q)=\rank(P) + \rank(Q).
\]
This set of perturbations is \emph{generic} in the sense that
it is open and dense in the set of pencils with given size and rank.
For such perturbations $Q$ it is shown that the number and the dimensions
of the Jordan blocks associated to an eigenvalue increase under perturbations of the above form.

The Theorem \ref{grosserWurf} can be interpreted in terms of the Kronecker invariants.

%Hence, the partial multiplicities of a matrix pencil coming from the Segre
%characteristic of the matrix entry $A_0$ in the
%first block of the Kronecker canonical form \eqref{kcf} correspond
%to the lengths of some Jordan chains.
%However, in addition, the column minimal indices of $P(s)$ also induce Jordan chains at each eigenvalue of $P$. The next result makes this more precise.

\begin{thm} \label{MarseilleMain}
Given a matrix pencil $P(s)=sE-F$ in Kronecker canonical form~\eqref{kcf}, assume that $\la\in\C$ is an eigenvalue of $P$. Then
\begin{equation}\label{Hauptergebnis}
\mathcal L_{\lambda}^j(P) = N\big((A_0-\lambda)^j\big) \oplus \{0\} \oplus N\big((K_\veps^{-1}L_\veps-\lambda)^{j}\big)\oplus \{0\},
\end{equation}
where the first $\{0\}$ in \eqref{Hauptergebnis} is in $\mathbb C^{|\alpha|}$ and the last $\{0\}$
is in $\mathbb C^{|\eta|}$.
\end{thm}

Note that in Theorem \ref{MarseilleMain} $K_\veps^{-1}L_\veps$ has to be interpreted as a linear relation.

\begin{proof}[Proof of Theorem \rmref{MarseilleMain}] First, assume that $(x_n, \ldots ,x_{0})$ in $\mathbb C^d$ is a Jordan chain at $\lambda\in\C$ for the matrix pencil $P$. According to Lemma~\ref{Marseille0} it is no restriction to assume $\lambda =0$.
Since $N_\alpha$ is in block diagonal form it is assumed without restriction
that $\alpha$ has only one entry. Let  $\veps$ and $\eta$ be multi-indices with $k$ and $l$ zeros.
We decompose the vectors $x_n, \ldots ,x_{0}$ according to the decomposition
$\C^d=\mathbb C^{n_0+\alpha + (|\veps|+n_\veps) +|\eta|}=
\mathbb C^{n_0}\oplus \mathbb C^{\alpha}\oplus \mathbb C^{|\veps|+n_\veps}\oplus \mathbb C^{|\eta|}$
corresponding to the Kronecker canonical form,
\begin{equation}\label{Wolfine}
x_j=(x_{j,1}\ x_{j,2}\ x_{j,3}\ x_{j,4})^\top \in \mathbb C^{n_0+\alpha+(|\veps|+n_\veps)+|\eta|}
\quad \mbox{for } j=0, \ldots, n.
\end{equation}

Consider  the third entry of~\eqref{Wolfine}.  By~\eqref{pehquh}, $\veps$ has the form
$$
\veps=(\veps_{1},\ldots, \veps_{n_\veps-k},0, \ldots, 0),
$$
where  $\veps_j\geq1$ for $j=1,\ldots, \veps_{n_\veps-k}$.
Then $K_\veps$ and $L_\veps$ are of the form
\[
K_\veps=\left[\begin{array}{ccccccc}
 K_{\veps_{1}} & & & & 0 & \cdots & 0\\
 & K_{\veps_{2}}& & &\vdots &&\vdots\\
&&\ddots& &\vdots &&\vdots \\
& & & K_{\veps_{n_\veps-k}}& 0 & \cdots & 0
\end{array}
\right]
\]
and
\[
L_\veps=\left[\begin{array}{ccccccc}
 L_{\veps_{1}} & & & & 0 & \cdots & 0\\
 & L_{\veps_{2}}& & &\vdots &&\vdots\\
&&\ddots& &\vdots &&\vdots \\
& & & L_{\veps_{n_\veps-k}}& 0 & \cdots & 0
\end{array}
\right]
\]
where the last $k$ columns in $K_\veps$ and in $L_\veps$ consist of zeros only. Hence, for $i\in \mathbb N$,
$$
N\big((K_\veps^{-1}L_\veps)^{i}\big) = \left(\bigoplus_{j=1}^{n_\veps-k}
N\big((K_{\veps_j}^{-1}L_{\veps_j})^{i}\big)\right)\oplus \mathbb C^k,
$$
and for the third entry of~\eqref{Wolfine} one finds
$$
x_{j,3} \in N\big((K_\veps^{-1}L_\veps)^{j+1}\big) \qquad \text{for $j=0,\ldots,n$}.
$$
This shows that it is sufficient to consider the case that $n_\veps=1$,

Now the fourth entry of~\eqref{Wolfine} is considered. By~\eqref{pehquh}, $\eta$ has the form
$$
\eta=(\eta_{1},\ldots, \eta_{n_\eta-l},0, \ldots, 0),
$$
where  $\eta_j\geq1$ for $j=1,\ldots, n_\eta-l$.
Thus $K_\eta^\top$ and $L_\eta^\top$ are of the form
\[
K_\eta^\top=\left[\begin{array}{ccc}
 K_{\eta_{1}}^\top &&\\
&\ddots&\\
&& K_{\eta_{n_\eta-l}}^\top\\
0&\hdots& 0\\
\vdots& & \vdots\\
0&\hdots& 0
\end{array}
\right],\quad
L_\eta^\top=\left[\begin{array}{ccc}
  L_{\eta_{1}}^\top &&\\
&\ddots&\\
&& L_{\eta_{n_\eta-l}}^\top\\
0&\hdots& 0\\
\vdots& & \vdots\\
0&\hdots& 0
\end{array}
\right],
\]
where the last $l$ rows in $K_\eta^\top$ and in $L_\eta^\top$ consist of zeros only.
%Note that the matrices $K_{\eta_{1}}^\top  \oplus \cdots \oplus K_{\eta_{n_\eta-l}}^\top$ and
%$L_{\eta_{1}}^\top  \oplus \cdots \oplus L_{\eta_{n_\eta}-l}^\top$ have full
%rank.
In order to show that the vectors $x_{0,4}, \ldots, x_{n,4}$ are zero, it remains to consider the case that $n_\eta=1$.

The considerations above have shown that we can restrict us to the case that $n_\alpha=n_\veps=n_\eta=1$.
Hence, in~\eqref{kcf} we have
$\alpha, \veps, \eta \in \mathbb N$ with
$$
\alpha \geq 1,\quad \veps\geq 1, \mbox{ and } \eta\geq 1.
$$
%We decompose the vectors $x_n, \ldots ,x_{0}$
%according to the decomposition
%\[
%\C^d=\mathbb C^{n_0+\alpha + (\veps +1) +\eta}=
%\mathbb C^{n_0}\oplus \mathbb C^{\alpha}\oplus \mathbb C^{\veps +1}\oplus \mathbb C^{\eta},
%\]
%corresponding to the Kronecker canonical form,
%$$
%x_j=(x_{j,1}\ x_{j,2}\ x_{j,3}\ x_{j,4})^\top \in \mathbb C^{n_0+\alpha+(\veps+1)+\eta}
%\quad \mbox{for } j=0, \ldots. n.
%$$
Since  $(x_n, \ldots ,x_{0})$ is a Jordan chain of $P$ at $\la=0$, the following equations are satisfied for $j=1,\ldots,n$:
\begin{eqnarray}\label{unno}
 A_0x_{0,1}=0,&  & A_0x_{j,1}=x_{j-1,1}, \\\label{doss}
 I_{\alpha}x_{0,2}=0, && I_{\alpha}x_{j,2}= N_\alpha x_{j-1,2},\\\label{tress}
 L_\veps x_{0,3}=0, && L_\veps x_{j,3}= K_\veps x_{j-1,3},\\\label{qutto}
 L_\eta^\top x_{0,4}=0, && L_\eta^\top x_{j,4}= K_\eta^\top x_{j-1,4}.
\end{eqnarray}
Thus, by \eqref{doss}, the vectors $x_{0,2}, \ldots, x_{n,2}$ are zero. Similarly, by  \eqref{qutto},
the vectors $x_{0,4}, \ldots, x_{n,4}$ are zero. Equation \eqref{unno} shows that
$(x_{n,1}, \ldots ,x_{0,1})$ is a Jordan chain at zero for the matrix $A_0$. Finally,
\eqref{tress} for $j=0,\ldots, n$ gives
$$
x_{j,3} \in N\big((K_\veps^{-1}L_\veps)^{j+1}\big).
$$
This shows that every vector in the chain $(x_n, \ldots ,x_{0})$ is an element in the right hand-side of~\eqref{Hauptergebnis}. Therefore, $\calL_\la^n(P)$ is contained in the right hand-side of~\eqref{Hauptergebnis}.

%Then equations \eqref{unno}--\eqref{qutto} from Case 1 hold and in the same way one concludes
%that $x_{0,2}, \ldots, x_{n,2}$ are zero and $(x_{n,1}, \ldots ,x_{0,1})$ is a Jordan chain at zero for the matrix $A_0$. In particular, for $j=0, \ldots, n$ we have
%$$
%x_{j,1} \in  N\big(A_0^{j+1}\big).
%$$

\medskip

Conversely, if $x_n$ is an element in the right hand-side of~\eqref{Hauptergebnis}
for $j=n+1$, and decomposing $x_n$ as in \eqref{Wolfine}, we have
  $x_n=(x_{n,1}\ 0\ x_{n,3}\ 0)^\top$ with
  $$
  x_{n,1}  \in N\big(A_0^{n+1}\big) \quad \mbox{and} \quad
    x_{n,3}  \in N\big((K_\veps^{-1}L_\veps)^{n+1}\big).
  $$
  Therefore, for each $i=0,\ldots,n-1$ there exist vectors $x_{i,1}$ and $x_{i,3}$ which satisfy equations
  \eqref{unno} and \eqref{tress}. For $i=0,\ldots, n-1$, set
  $$
  x_i:=(x_{i,1}\ 0\ x_{i,3}\ 0)^\top.
  $$
  From this, it is easy to see that
  $(x_n, \ldots ,x_{0})$ is a Jordan chain (at $\la=0$) for the matrix pencil $P$. In particular, $x_n\in \calL_\la^{n+1}(P)$.
\end{proof}

Using the above result, we present an alternative version of Theorem \ref{grosserWurf} in terms of the Weyr characteristics of the Kronecker canonical form. For simplicity, we state
it here only for finite eigenvalues $\la$. A similar statement can be shown for
$\la = \infty$ applying Corollary \ref{infinito}.

\begin{thm}\label{grosserWurf2}
Let $\lambda \in \mathbb C$ and $n\in\N\cup\{0\}$.
Given $P(s)=sE-F$ in $\C^{d\times d}$, let $Q$ be a rank-one matrix pencil as in \eqref{stoerung}.
Assume that $A_0$ and $\widetilde{A_0}$ are the matrices in Jordan canonical form appearing in the Kronecker canonical forms \eqref{kcf} of $P$ and $P+Q$,
%and $\widetilde{\veps}=(\widetilde{\veps}_1,\ldots, \widetilde{\veps}_{n_{\widetilde \veps}})$ are the column minimal indices of $P+Q$.
denote by $w_n(\la)$ and $\widetilde{w}_n(\la)$ the Weyr characteristics
of the matrices  $A_0$ and $\widetilde{A}_0$, according to \eqref{WeyrWeyr}
and by $\Delta_n$ and $\widetilde{\Delta}_n$ the Weyr characteristics of the column minimal indices of $P$ and $P+Q$ according to \eqref{WeyrWeyrWeyr}. Then the following statements hold:
\begin{itemize}
\item[\rm (i)] If both pencils $P$ and $P+Q$ are regular, then $\Delta_n=\widetilde{\Delta}_n=0$ and
	\begin{align*}
	\left|\widetilde{w}_{n+1}(\la)-w_{n+1}(\la)\right| \leq 1.
	\end{align*}
\item[\rm (ii)]
 If $P$ is regular and $P+Q$ is singular, then $\Delta_n=0$ and
	\begin{align*}
-1-n \leq\widetilde{w}_{n+1}(\la)-w_{n+1}(\la) \leq 1 + \widetilde{\Delta}_n.
	\end{align*}
\item[\rm (iii)]
 If $P$ is singular and $P+Q$ is regular, then $\widetilde{\Delta}_n=0$ and
	\begin{align*}
-1 -\Delta_n \leq   \widetilde{w}_{n+1}(\la)-w_{n+1}(\la)\leq n+1.
	\end{align*}
\item[\rm (iv)]
 If both $P$ and $P+Q$ are singular, then
	\begin{align*}
\left|\widetilde{w}_{n+1}(\la)-w_{n+1}(\la) + \widetilde{\Delta}_n -\Delta_n  \right| \leq n+1.
	\end{align*}
\end{itemize}
\end{thm}

\begin{proof}
Note that if $S$ and $T$ are invertible matrices and $(x_n,\ldots, x_0)$  is a Jordan chain of some pencil $P(s)$, the Definition~\ref{LA} immediately implies that
 $(T^{-1}x_n,\ldots, T^{-1}x_0)$ is a Jordan chain of the pencil $\hat P(s)=SP(s)T$. Hence $\dim {\mathcal{L}_{\la}^n(P)}=\dim {\mathcal{L}_{\la}^n(\hat P)}$ for all
$n\in\N\cup\{0\}$.
According to Lemma~\ref{Marseille0}, if $\la\in\C$ we may assume $\lambda = 0$.
As a consequence of Theorem~\ref{MarseilleMain},
$$
\dim\frac{\mathcal{L}_{0}^{n+1}(P)}{\mathcal{L}_{0}^n(P)}=
\dim\frac{N (A_0^{n+1})}{N(A_0^n)}
+\dim \frac{N(K_\veps^{-1}L_\veps)^{n+1}}{N(K_\veps^{-1}L_\veps)^{n}},
$$
and it is straightforward to see that
$$
\dim \frac{N(K_\veps^{-1}L_\veps)^{n+1}}{N(K_\veps^{-1}L_\veps)^{n}} = \Delta_n.
$$
The same holds for the pencil $P+Q$, and we obtain
$$
\dim\frac{\mathcal{L}_{0}^{n+1}(P+Q)}{\mathcal{L}_{0}^n(P+Q)}=
\dim\frac{N (\widetilde{A}_0^{n+1})}{N(\widetilde{A}_0^n)}
+\widetilde{\Delta}_n.
$$
Then, the result follows immediately from Theorem \ref{grosserWurf}.
\end{proof}

Finally, we compare the above result with Section 4 in \cite{DD07}. A particular case of Lemma 4.2 in \cite{DD07} can be restated in the following way. Given a matrix pencil $P(s)$  in $\C^{d\times d}$ with $\rank(P)=r$, assume that $\la$ is an eigenvalue of $P$ with partial multiplicities $0\leq m_1\leq\ldots\leq m_r$. Let $Q(s)$ be a matrix pencil in $\C^{d\times d}$ with $\rank(Q)=1$ and let $m$ be the partial multiplicity of $\la$ relative to $Q$ ($m$ can also be zero). If $\rank(P+Q)=r+1$ and $m_i<m\leq m_{i+1}$ for some $i=0,1,\ldots,r$ (where $m_0=-1$ and $m_{r+1}=\infty)$, then the partial multiplicities $0\leq m'_1\leq \ldots \leq m'_{r+1}$ of $\la$ relative to $P+Q$ satisfy
\begin{equation}\label{DTD}
m'_1=m_1,\ \ldots,\ m'_i=m_i, \quad m'_{i+1}\geq m, \quad m'_{i+2}\geq m_{i+1},\ \ldots,\ m'_{r+1}\geq m_r.
\end{equation}

%If we denote by $A_0$ and $\widetilde{A_0}$ the matrices in Jordan canonical form appearing in the Kronecker canonical forms of $P$ and $P+Q$, respectively, then \eqref{DTD} can be rephrased as
%\begin{align*}
%\dim\frac{N(\widetilde{A_0}-\lambda)^{n+1}}{N(\widetilde{A_0}-\lambda)^n}=
%\dim\frac{N(A_0-\lambda)^{n+1}}{N(A_0-\lambda)^n}
% +1, \quad &\text{if $n=0,\ldots, m-1$, and} \\
%\dim\frac{N(\widetilde{A_0}-\lambda)^{n+1}}{N(\widetilde{A_0}-\lambda)^n}=
%\dim\frac{N(A_0-\lambda)^{n+1}}{N(A_0-\lambda)^n}, \quad &\text{if $n=m,\ldots,m'_{r+1}-1$}.
%\end{align*}

Given matrix pencils $P$ and $Q$ as in Theorem \ref{grosserWurf2}, in order to satisfy the conditions of item (iii) it is necessary that $\rank(P)=d-1$ and $\rank(P+Q)=d$. Hence, the hypothesis $\rank(P+Q)=\rank(P) + \rank(Q)$ is fulfilled and Lemma 4.2 in \cite{DD07} provides a better estimate. Reversing the roles of $P$ and $P+Q$, the same happens with item (ii).

However, if both $P$ and $P+Q$ are singular the result in item (iv) holds, independently of the hypothesis $\rank(P+Q)=\rank(P) + \rank(Q)$. Therefore, Theorem \ref{grosserWurf2} gives new information in case that $\rank(P+Q)\neq\rank(P) + \rank(Q)$.

\medskip\medskip
\noindent
{\bf Acknowledgment}
\\\\
%We wish to thank the referees for fruitful hints.
The authors would like to thank the reviewers for their careful reading of the manuscript and for several comments which helped to improve this work.

\end{document}